\documentclass[a4paper,11pt]{article}
\usepackage[sc]{mathpazo}
\usepackage[T1]{fontenc}
\usepackage{amsthm,amsmath,amssymb,amsfonts,hyperref,a4wide}
\usepackage[usenames,dvipsnames]{xcolor} 
\usepackage{tikz}
\makeatletter
\tikzset{my loop/.style =  {to path={
  \pgfextra{}
  [looseness=12,min distance=6mm]
  \tikz@to@curve@path},font=\sffamily\small
  }}  
\makeatletter 
\usetikzlibrary{matrix,arrows,backgrounds,positioning,fit,shapes}

\newtheorem{theorem}{Theorem}[section]
\newtheorem{lemma}[theorem]{Lemma}
\newtheorem{cor}[theorem]{Corollary}

\newtheorem{proposition}[theorem]{Proposition}
\theoremstyle{definition}

\newtheorem{df}{Definition}[section]
\theoremstyle{remark}
\newtheorem{rem}{Remark}[section]
\newtheorem{ques}{Question}[section]

\newcommand*{\R}{\mathbb{R}}
\newcommand*{\Z}{\mathbb{Z}}

\newcommand*{\N}{\mathbb{N}}
\newcommand*{\C}{\mathbb{C}}

\newcommand*{\F}{\mathcal{F}}
\newcommand*{\G}{\mathcal{G}}

\newcommand*{\ind}{\text{ind}}

\newcommand*{\eps}{\varepsilon}

\title{Deterministic polynomial-time approximation algorithms for partition functions and graph polynomials\footnote{An extended abstract of this work has been accepted in Eurocomb 2017.}}
\author{Viresh Patel\footnote{Korteweg de Vries Institute for Mathematics, University of Amsterdam. Email: \texttt{vpatel@uva.nl}. Supported by the Netherlands Organisation for Scientific Research (NWO) through the Gravitation Programme Networks (024.002.003).} \and Guus Regts\footnote{Korteweg de Vries Institute for Mathematics, University of Amsterdam. Email: \texttt{guusregts@gmail.com}. Supported by a personal NWO Veni grant}.}
\begin{document}
\maketitle
\abstract{In this paper we show a new way of constructing deterministic polynomial-time approximation algorithms for computing complex-valued evaluations of a large class of graph polynomials on bounded degree graphs. In particular, our approach works for the Tutte polynomial and  independence polynomial, as well as partition functions of complex-valued spin and edge-coloring models.

More specifically, we define a large class of graph polynomials $\mathcal C$ and show that if $p\in \cal C$ and there is a disk $D$ centered at zero in the complex plane such that $p(G)$ does not vanish on $D$ for all bounded degree graphs $G$, then for each $z$ in the interior of $D$ there exists a deterministic polynomial-time approximation algorithm for evaluating $p(G)$ at $z$.
This gives an explicit connection between absence of zeros of graph polynomials and the existence of efficient approximation algorithms, allowing us to show new relationships between well-known conjectures.

Our work builds on a recent line of work initiated by Barvinok \cite{B14a, B14b, B15, B16}, which provides a new algorithmic approach besides the existing Markov chain Monte Carlo method and the correlation decay method for these types of problems.

\noindent \begin{footnotesize}

\end{footnotesize}
}

\section{Introduction}
Computational counting is an important area of computer science where one seeks to find efficient algorithms to count certain combinatorial objects such as independent sets, proper colorings, or matchings in a graph. More generally, each combinatorial counting problem has an associated generating function, namely the independence polynomial for independent sets, the chromatic and more generally Tutte polynomial for proper graph colorings, and the matching polynomial for matchings. Such graph polynomials are studied in mathematics and computer science, but also  in statistical physics where they are normally referred to as partition functions. A fundamental question asks for which graphs and at which numerical values one can approximately evaluate these polynomials efficiently.
Indeed the counting problems correspond to evaluating these graph polynomials or partition functions at particular values.

Many of these counting problems
are known to be computationally hard in the sense of being \#P-hard, even when one restricts to graphs of maximum degree at most three \cite{BDGJ99,DG00a}.
On the other hand several efficient randomized approximation algorithms exist for some of these \#P-hard problems via the use of the powerful Markov chain Monte Carlo technique. In a major breakthrough, Weitz \cite{W6}, inspired by ideas from statistical physics, developed the so-called correlation decay method allowing him to obtain the first efficient deterministic approximation algorithm for counting independent sets in graphs of maximum degree at most five. (One expects no such algorithm for graphs of maximum degree larger than five \cite{SS12}, while previously the best known (randomized) algorithm worked only for graphs of maximum degree at most four.)
The correlation decay method has subsequently been refined and applied to various other problems; see e.g. \cite{BGKNT7,GK12,LY13,SST14} and references therein.

In this paper we consider a different approach.
The approach is quite robust in that it can be applied to a large class of graph polynomials and gives the first general polynomial-time method to approximate graph polynomials at complex values for bounded degree graphs. Very recently complex evaluations have also been considered by Harvey, Srivastava, and Vondr{\'a}k~\cite{HSV16} for the special case of the independence polynomial.
Complex evaluations of graph polynomials, aside from being the natural extensions of real evaluations, arise as interesting counting problems e.g.\ counting restricted tensions or flows can be modelled as the partition functions of a complex spin system (see \cite{GGN14}) and the number of homomorphisms into any fixed graph can be modelled as the partition function of a complex edge-coloring model (see \cite{S7,S10}).

A further important aspect of our work is to highlight the explicit relation between the (absence of complex) roots of a graph polynomial and efficient algorithms to evaluate it. Indeed, in Remark~\ref{re:sokal} below we give the explicit connection between a conjecture of Sokal on zero-free regions of the chromatic polynomial and the notorious algorithmic problem of efficiently approximating the number of proper colorings in a bounded degree graph.

Our approach combines a number of ingredients including ideas from sparse graph limits \cite{CF12}, results on the locations of zeros of graph polynomials and partition functions \cite{S98,SS05,JPS13,BS14a,BS14b,R15} and an algorithmic development due to Barvinok \cite{B14a}.
The Taylor approximation technique of Barvinok has been used to construct deterministic quasi-polynomial-time approximation algorithms 
for evaluating a number of graph partition functions (for general graphs); see e.g.\ work by Barvinok \cite{B14a,B14b,B15,B16}, by Barvinok and Sobe\'ron \cite{BS14a,BS14b}, and by the second author \cite{R15}. We refer to Barvinok's recent book \cite{Babook} for more background.

The approach can be roughly described as follows. 
First the problem of evaluating the partition function or graph polynomial is cast as the evaluation of a univariate polynomial. Next, a region is identified where this polynomial does not vanish; hence in this region the logarithm of the polynomial is well-approximated by a low-order Taylor approximation (of order $\log n$, where $n$ in the degree of the polynomial). Finally we must compute this Taylor approximation by efficiently computing the first $O(\log n)$ coefficients of the polynomial. 
So far this approach has only resulted in algorithms that run in quasi-polynomial time.
The main technical contribution of the present paper is a polynomial-time algorithm for computing (essentially) the first $O(\log n)$ coefficients of a large class of graph polynomials whenever we work with bounded degree graphs cf.\ Theorem~\ref{thm:compute coef}, and we believe it to be of independent interest.

Below we shall state and discuss some concrete results that can be obtained by combining this approach with (known) results on the location of roots of graph polynomials and partition functions.
In particular, we obtain new deterministic polynomial-time algorithms (FPTAS) for evaluating the independence polynomial, the Tutte polynomial, and computing partition functions of spin and edge-coloring models in the case of bounded degree graphs.
Before we state our algorithmic results, we first need a definition.
Since we will approximate polynomials at complex values, we define what it means to be a good approximation.
\begin{df}
Let $q$ and $\xi$ be a non-zero complex numbers. We call $\xi$ a \emph{multiplicative $\eps$-approximation to $q$} if 
$e^{-\eps}\leq |q|/|\xi|\leq e^{\eps}$ and if the angle between $\xi$ and $q$ (as seen as vectors in $\C=\R^2$) is at most $\eps$.
\end{df}

\subsection{The independence polynomial}\label{subsec:independence}
The \emph{independence polynomial} of a graph $G=(V,E)$ is denoted by $Z(G)$ and is defined as
\begin{equation}\label{eq:def ind pol}
Z(G)(\lambda):=\sum_{\substack{I\subseteq V\\ I \text{ independent}}}\lambda^{|I|}.
\end{equation}
In \cite{W6} Weitz proved, based on the correlation decay method, that if $0\leq \lambda<\lambda_c$,
where
\[
\lambda_c=\frac{(\Delta-1)^{\Delta-1}}{(\Delta-2)^\Delta},
\]
then there exists a deterministic algorithm, which given a graph $G=(V,E)$ of maximum degree at most $\Delta$ and $\eps>0$, computes a multiplicative $\eps$-approximation to $Z(G)(\lambda)$ in time $(|V|/\eps)^{O(1)}$.
Sly and Sun \cite{SS12} proved this is tight by showing that, as soon as $\lambda>\lambda_c$, one cannot efficiently approximate $Z(G,\lambda)$ unless NP=RP.

In Section \ref{sec:independence} we prove the following result, which has been independently obtained by Harvey, Srivastava and Vondr\'ak \cite{HSV16} using the correlation decay method.
\begin{theorem}\label{thm:ind general}
Let $\Delta\in \N$ and let $\lambda\in \C$ be such that $|\lambda|<\lambda^*(\Delta):=\frac{(\Delta-1)^{\Delta-1}}{\Delta^{\Delta}}$. Then there exists a deterministic algorithm, which, given a graph $G=(V,E)$ of maximum degree at most $\Delta$ and $\eps>0$, computes a multiplicative $\eps$-approximation to $Z(G)(\lambda)$ in time $(|V|/\eps)^{O(1)}$.
\end{theorem}
\begin{rem}
From the proof of Theorem~\ref{thm:ind general} it follows that the running time is in fact bounded by 
\[
(|V|/\eps)^{\frac{D}{1 -|\lambda|/\lambda^*(\Delta)}\ln( \Delta) }|V|^{D'}
\]
for some absolute constants $D, D'$.
\end{rem}
Theorem~\ref{thm:ind general} in fact also applies to the multivariate independence polynomial, as we will briefly explain in Subsection~\ref{subsec:mult}.

For positive valued $\lambda$ our result is weaker than Weitz's result since $\lambda_c > \lambda^*$. However our result works for negative\footnote{In an unpublished note \cite{Sr15} Srivastava notes that the correlation decay method of Weitz in fact also applies to negative $\lambda$ as long as $\lambda>-\lambda^*$.}  and even complex  $\lambda$.
The case $\lambda<0$ is relevant due to its connection to the Lov\'asz local lemma, cf.\ \cite{SS05}.
We remark here that by very recent results of Peters and the second author \cite{PR17} confirming a conjecture of Sokal \cite{S1}, and by the method from Subsection~\ref{subsec:clawfree} below, we are able to obtain an alternative proof of Weitz's result. We say more about this in Section~\ref{sec:conclude}.

The value $\lambda^*$ in Theorem \ref{thm:ind general} originates from a paper of  Shearer \cite{S98} (see also Scott and Sokal \cite{SS05}) where it is shown that for graphs of maximum degree $\Delta$, the independence polynomial does not vanish at any $\lambda \in \mathbb{C}$ satisfying $|\lambda|\leq \lambda^*$. 
Also the value of $\lambda^*$ is tight, as there exists a sequence of trees $T_n$ of maximum degree at most $\Delta$ and $\lambda_n<-\lambda^*$ with $\lambda_n\to -\lambda^*$ such that $Z(T_n,\lambda_n)=0$, cf.\ \cite[Example 3.6]{SS05}.
Theorem~\ref{thm:ind general} is also tight, as very recently, Galanis, Goldberg and \v{S}tefankovi\v{c} showed that it is NP-hard to approximate $Z_G(\lambda)$ when $\lambda<-\lambda^*$.

As an extension to Theorem \ref{thm:ind general}, we are able to efficiently approximate the independence polynomial on almost the entire complex plane for the special class of claw-free graphs. We make use of a result of Chudnovsky and Seymour \cite{CS7}
stating that the independence polynomial of a claw-free graph has only negative real roots. 
We prove the following result in Subsection \ref{subsec:clawfree}.
\begin{theorem}\label{thm:ind claw free}
Let $\Delta\in \N$ and let $\lambda\in \C$ be such that $\lambda$ is not a real negative number. Then there exists a deterministic algorithm, which, given a claw-free graph $G=(V,E)$ of maximum degree at most $\Delta$ and $\eps>0$, computes a multiplicative $\eps$-approximation to $Z(G)(\lambda)$ in time $(|V|/\eps)^{O(1)}$.
\end{theorem}
Note that when $G$ is the line graph of some graph $H$ we have that $Z_G(\lambda)$ is equal to the matching polynomial of $H$. 
So in particular, Theorem \ref{thm:ind claw free} implies a result of Bayati, Gamarnik, Katz, Nair, and Tetali \cite{BGKNT7}.
Our proof of it however is entirely different from the proof in \cite{BGKNT7}.

\subsection{The Tutte polynomial}\label{subsec:Tutte}
The random cluster formulation of the Tutte polynomial of a graph $G=(V,E)$ is a two-variable polynomial, which is denoted  by $Z_T(G)$ and is defined by
\begin{equation}
Z_T(G)(q,w):=\sum_{A\subseteq E} q^{k(A)}w^{|A|},
\end{equation}
where $k(A)$ denotes the number of components of the graph $(V,A)$.
In particular, if $w=-1$, $Z_T(G)(q,-1)$ is equal to the chromatic polynomial of $G$.

Jerrum and Sinclair \cite{JS93} showed that when $q=2$ and $w>0$ there exists a randomized polynomial-time approximation algorithm for computing evaluations of the Tutte polynomial in general.
In \cite{GJ12} Goldberg and Jerrum showed that approximating evaluations of the Tutte polynomial on general graphs for $q>2$ and $w>0$ is as hard as counting independent sets in bipartite graphs and in \cite{GJ12a} Goldberg and Jerrum showed that for several choices of real parameters $(q,w)$ it is even \#P-hard to approximate the evaluation of the Tutte polynomial on general graphs.
Goldberg and Guo \cite{GG14} looked at the complexity of approximately evaluating the Tutte polynomial for general graphs at complex values.

When $w=-1$ and $q\in \N$, $Z_T(G)(q,w)$ gives the number of $q$-colorings of $G$. Lu and Yin \cite{LY13} showed that when $q>2.58\Delta$ there exists a deterministic polynomial-time algorithm for approximating the Tutte polynomial at $(q,-1)$ on graphs of maximum degree at most $\Delta$. There are many randomized algorithms of the sort above with sharper bounds on $q$; see e.g.\ Jerrum \cite{J95} and Vigoda \cite{V00}.
As far as we know there are no general results known for the Tutte polynomial on bounded degree graphs.

We will consider the Tutte polynomial as a univariate polynomial by considering $w$ to be constant.
In Section \ref{sec:Tutte} we prove the following result.
\begin{theorem}\label{thm:Tutte}
Let $\Delta\in \N$ and let $w\in \C$. Then there exists a constant $K$ (depending on $\Delta$ and $w$) such that if $q\in \C$ is such that $|q|>K$, then there exists a deterministic  algorithm, which, given a loopless multigraph $G=(V,E)$ of maximum degree at most $\Delta$ and $\eps>0$, computes a multiplicative $\eps$-approximation to $Z(G)(q,w)$ in time $(|V|/\eps)^{O(1)}$.
\end{theorem}
\begin{rem}
From the proof the Theorem~\ref{thm:Tutte} it follows that the running time is in fact bounded by 
\[
(|V|/\eps)^{\frac{D}{1 -K/|q|}\Delta \ln(\Delta)}|V|^{D'}
\]
for some absolute constants $D,D'$.

The constant $K$ in the theorem above comes from a paper of Jackson, Procacci and Sokal \cite{JPS13} and unfortunately takes half a page to state exactly. 
However, when $w$ satisfies $|1+w|\leq 1$ (this includes the chromatic polynomial), the constant $K$ may be taken to be $6.91\Delta$.
\end{rem}
\begin{rem} \label{re:sokal}
Sokal \cite[Conjecture 21]{J06} conjectured that $Z_T(G)(q,-1)\neq 0$ as long as $\Re(q)> \Delta(G)$.
Combined with our results (and the technique from Section~\ref{subsec:clawfree}) a confirmation of the conjecture would imply an efficient approximation algorithm for computing the number of $(\Delta+1)$-colorings of any graph $G$ of maximum degree at most $\Delta$, a notorious problem in computational counting.
\end{rem}

\subsection{Partition functions of spin models}\label{subsec:spin}
Let $A\in \C^{k\times k}$ be a symmetric matrix. In the context of statistical physics $A$ is often called a \emph{spin model} cf.\ \cite{HJ}.
For a graph $G=(V,E)$, the partition function of $A$ is defined as
\begin{equation}
p(G)(A)=\sum_{\phi:V\to [k]}\prod_{\{u,v\}\in E}A_{\phi(u),\phi(v)} \label{eq:part spin}.
\end{equation}
If $A$ is the adjacency matrix of some graph $H$, then $p(G)(H)$ is equal to the number of graph homomorphisms from $G$ to $H$.
In \cite{BS14a} $p(G)(A)$ is called the graph homomorphism partition function.

Building on a line of research started by Dyer and Greenhill \cite{DG00} and Bulatov and Grohe \cite{BG05}, a full dichotomy theorem has been proved for the complexity of exactly computing the partition function of a complex spin model by Cai, Chen and Lu \cite{CCL13}. This dichotomy essentially says that computing the partition function of $A$ exactly is \#P hard unless the matrix $A$ has some special structure.

Lu and Yin \cite{LY13} proved, using the correlation decay approach, that for fixed $\Delta\in \N$, if a real matrix $A$ is sufficiently close to the all ones matrix (i.e. $|A_{i,j}-1|\leq O(1)/\Delta$ for all $i,j=1,\ldots k$), then there exists a $(|V(G)|/\eps)^{O(1)}$-time algorithm for computing a multiplicative $\eps$-approximation to $P(G)(A)$ on graphs of maximum degree at most $\Delta$. 
Barvinok and Sob\'eron \cite{BS14a} showed that there exists a $(|V(G)|/\eps)^{O(\ln|V(G)|)}$-time algorithm for complex-valued matrices $A$ that satisfy $|A_{i,j}-1|\leq O(1)/\Delta$ for all $i,j=1,\ldots,k$.

Building on the work of Barvinok and Sob\'eron we prove in Section \ref{sec:hom} the following result.
\begin{theorem}\label{thm:part spin}
Let $\Delta,k\in \N$. Then there exists a deterministic algorithm, which, given a graph $G=(V,E)$ of maximum degree at most $\Delta$, a (complex-valued) symmetric $k\times k$ matrix $A$ such that  $|A_{i,j}-1|\leq 0.34/\Delta$ for all $i,j=1,\ldots,k$, and $\eps>0$, computes a multiplicative $\eps$-approximation to $p(G)(A)$ in time $(|V|/\eps)^{O(1)}$.
\end{theorem}
\begin{rem}
The constant $0.34$ can be replaced by $0.45$ if $\Delta\geq 3$, and by $0.54$ if  $\Delta$ is large enough, cf.\ \cite{BS14a}. 
\end{rem}

In \cite{BS14b} Barvinok and Sober\'on introduced partition functions of graph homomorphisms of $G$ with multiplicities and gave a quasi-polynomial-time algorithm for computing them for certain matrices. 
In Section \ref{sec:hom} we will show that our results also apply to these partition functions.

\subsection{Partition functions of edge-coloring models}\label{subsec:edge}
Edge-coloring models originate in statistical physics and their partition functions have been introduced to the graph theory community by de la Harpe and Jones \cite{HJ} (where they are called vertex models).
We call any map $h:\N^k\to \C$ a \emph{$k$-color edge-coloring model}. 
For a graph $G=(V,E)$, the \emph{partition function} of $h$ is defined by 
\begin{equation}\label{eq:def pf}
p(G)(h):=\sum_{\phi:E\to[k]}\prod_{v\in V}h(\phi(\delta(v))),
\end{equation}
where $\delta(v)$ denotes the set of edges incident with the vertex $v$ and $\phi(\delta(v))$ denotes the multiset of colors that the vertex $v$ `sees', which we identify with its incidence vector in $\N^k$ so that we can apply $h$ to it. Explicitly, $\phi(\delta(v))$ is identified with the vector $(a_1, \ldots, a_k) \in \N^k$ if for each $i$ there are $a_i$ occurrences of the color $i$ amongst the edges incident with $v$.

Partition functions of edge-coloring models form a rich class of graph parameters including the number of matchings (take $h:\N^2\to \C$ defined by $h(\alpha)=1$ if $\alpha_1\leq 1$ and $0$ otherwise), as well as partition functions of spin models, as has been proved by Szegedy \cite{S7,S10}.
These partition functions can be seen as Holant problems; see e.g. \cite{CHL10,CLX11,CGW13}. 
They can also be seen as tensor network contractions. We refer the reader to \cite{R13} for more background.

Just as for partition functions for spin models much work has been done to establish a complexity dichotomy result for exactly computing Holant problems; see \cite{CHL10,CLX11,CGW13}.
Not much is known about the complexity of approximating partition functions of edge-coloring models except for a few special cases. 
As already mentioned, Bayati, Gamarnik, Katz, Nair, and Tetali \cite{BGKNT7} found an efficient approximation algorithm for counting matchings in bounded degree graphs and Lin, Liu and Lu \cite{LLL14} found efficient approximation algorithms for counting edge covers. Both of these algorithms are based on the correlation decay method.

Building on work of the second author \cite{R15} we will prove the following result in Section \ref{sec:edge}.
\begin{theorem}\label{thm:part edge}
Let $\Delta,k\in \N$. Then there exists a deterministic algorithm, which, given a multigraph $G=(V,E)$ of maximum degree at most $\Delta$, a $k$-color edge-coloring model $h$ such that $|h(\phi)-1|\leq 0.35/(\Delta+1)$ for all $\phi\in \N^k$, and $\eps>0$, computes a multiplicative $\eps$-approximation to $p(G)(h)$ in time $(|V|/\eps)^{O(1)}$.
\end{theorem}
\begin{rem}
The constant $0.35$ may be replaced by $0.47$ if $\Delta\geq 3$ and by $0.56$ if $\Delta$ is large enough; see \cite{R15}.
Moreover, for readers familiar with the orthogonal group invariance of these partition functions, it is interesting to note that one can use Corollary 6b from \cite{R15} to find a much larger family of edge-coloring models for which the partition function can be efficiently approximated.
\end{rem}

\subsection{Organization}
In the next section we shall consider an algorithm due to Barvinok \cite{B14a} to approximate evaluations of polynomials.
Section \ref{sec:coef} contains our main technical contribution: we will introduce a class of graph polynomials and give an efficient algorithm for computing their low order coefficients on bounded degree graphs.
These two algorithms (or variations of them) will then be combined in Sections \ref{sec:independence}--\ref{sec:edge} to prove the results above. These sections can be read independently of one another. 
Finally, we conclude in Section \ref{sec:conclude} with some remarks and questions.

\section{Approximating evaluations of polynomials}\label{sec:algorithms}
In this section we present an algorithm due to Barvinok \cite{B14a} to approximate evaluations of polynomials.
We take a slightly different approach and give full details for the sake of completeness.

Let $p\in \C[z]$ be a polynomial $p(z) = a_0 + a_1z + \cdots + a_dz^d$ of degree $d$ and suppose that $p(z)\neq 0$ for all $z$ in an open disk $D$ of radius $M$.
Define the function $f$ on this disk by
\begin{equation}
f(z):=\ln p(z),
\end{equation}
(where we fix a branch of the logarithm by fixing the principal value of the logarithm at $p(0)$). Recall by Taylor's Theorem that for each $t \in D$, $f(t) = \sum_{j=0}^\infty \frac{t^j}{j!}f^{(j)}(0)$.
In order to approximate $p$ at $t\in D$, we will find an additive approximation to $f$ at $t$ by truncating the Taylor expansion around $z=0$. For each $m \in \N$, let
\begin{equation}
T_m(f)(t):=f(0)+\sum_{j=1}^m \frac{t^j}{j!}f^{(j)}(0).
\label{eq:taylor}
\end{equation}

This can then be transformed to give a multiplicative approximation to $p$. 
It will be more convenient for us to use a slightly different form of (\ref{eq:taylor}) which we derive below.

Let $\zeta_1, \ldots, \zeta_d\in \C$ be the roots of $p$. 
Then we can write $p(z) = a_d(z - \zeta_1) \cdots (z-\zeta_d)$ and $f(z) = \ln(a_d)+\ln(z - \zeta_1) + \cdots +\ln(z - \zeta_d)$. From this we see that $f'(z) = (z - \zeta_1)^{-1} + \cdots + (z - \zeta_d)^{-1}$ and hence for $j\geq 1$,
\[
f^{(j)}(0) = -(j-1)!\sum_{i=1}^d \zeta_i^{-j}.
\]
Thus defining the $j$th \emph{inverse power sum} to be $p_j:= \zeta_1^{-j} + \cdots + \zeta_d^{-j}$ we see that \begin{equation}
\label{eq:TayPowSum}
T_m(f)(t) = f(0) -  \sum_{j=1}^m \frac{p_jt^j}{j} = \ln(a_0) - \sum_{j=1}^m \frac{p_jt^j}{j}.
\end{equation}
In the next proposition we derive a variant of the Newton identities that relate the inverse power sums and the coefficients of the polynomial.
\begin{proposition}
\label{pr:newton}
For the polynomial $p(z) = a_0 + \cdots + a_dz^d$ as above and its inverse power sums $p_j$ as defined above, we have for each $k=1,2, \ldots$ that
\[
ka_k= -\sum_{i=0}^{k-1}a_ip_{k-i}.
\]
(Here we take $a_i=0$ if $i>d$.)
\end{proposition}
\begin{proof}
From \eqref{eq:TayPowSum} we know that for $z\in D$ we have $\ln(p(z)) = \ln(a_0)-\sum_{j=1}^\infty \frac{p_jz^j}{j}$. 
Differentiating both sides and multiplying by $p(z)$ we obtain
\[
p'(z)=-p(z)\sum_{j=1}^\infty p_jz^{j-1}
\]
and so
\[
\sum_{k=1}^d ka_kz^{k-1} = -\sum_{i=0}^da_iz^i \sum_{j=1}^\infty p_jz^{j-1}.
\]
Comparing coefficients of $z^{k-1}$ on each side gives the desired identity.
\end{proof}

The next lemma shows that the quality of the approximation \eqref{eq:taylor} and hence \eqref{eq:TayPowSum} depends on the location of the complex roots of $p$.

\begin{lemma}\label{lem:approx}
Given $M>0$ and $t \in \C$ satisfying $|t|<M$, there exists a constant $C=C(t,M)\leq (1-|t|/M)^{-1}$ such that the following holds.
Suppose $p$ is a polynomial of degree $d$ with no roots in the open disk $D$ of radius $M$. Then for every $\varepsilon > 0$,  $\exp(T_m(f)(t))$ is a multiplicative $\varepsilon$-approximation to $p(t)$, where $m = C\ln(d / \varepsilon)$.
\end{lemma}
\begin{proof}
Let $q:=|t|/M$. Then, as $|t|<M$, we have $q<1$.
We will first show that 
\begin{equation}
\left| f(t)-T_m(f)(t) \right|\leq\frac{d q^{m+1}}{(m+1)(1-q)}.	\label{eq:approx}
\end{equation}
By \eqref{eq:TayPowSum} we have 
\begin{equation}\label{eq:Taylor 1}
|f(t)-T_m(f)(t)|\leq \left| \sum_{j=m+1}^{\infty}\frac{p_j t^j}{j} \right|   \leq \frac{1}{m+1}\sum_{j=m+1}^{\infty}|p_jt^j|.
\end{equation}
By assumption we know that $|\zeta_i|\geq M$ for each $i=1,\ldots,d$.
Hence $|p_j|\leq d/M^j$ and so $|p_jt^j|\leq dq^j$.
Substituting this in into \eqref{eq:Taylor 1} and using that $q<1$ we obtain \eqref{eq:approx}.

Take $m=C\ln (d/\varepsilon)$, where $C$ is chosen such that $C\geq (\ln1/q)^{-1}$ and $1/m \leq 1-q$ (so it is easy to check that e.g.\ $C =(1-q)^{-1}$ suffices).
Then the right-hand side of \eqref{eq:approx} is at most $\eps$. 
Write $z=T_m(f)(t)$. Then we have
{$|e^{f(t)-z}|\leq e^{|f(t)-z|}\leq e^{\eps}$}
and {similarly} $|e^{z-f(t)}|\leq e^{\eps}$. 
(This follows from the fact that for a complex number $y=a+bi$, we have $|e^y|=e^a \leq e^{|y|}$.)
Moreover, {the angle between $e^{z}$ and $e^{f(t)}$ is bounded by $|\Im \ln e^{z-f(t)}|\leq |\ln e^{z-f(t)}|\leq \eps$.}
This shows that $e^z=\exp(T_m(f)(t))$ is a multiplicative $\eps$-approximation to $p(t)$.
\end{proof}

From (\ref{eq:TayPowSum}) and Lemma~\ref{lem:approx}, if we have an efficient way of computing the inverse power sums $p_j$ from $j=1$ up to $O(\ln(\deg(p)))$ (which by Proposition~\ref{pr:newton} is essentially equivalent to computing the first $O(\ln(\deg(p)))$ coefficients of $p$), then we have an efficient way of approximating evaluations of $p$ at points in the disk around zero where $p$ is nonvanishing. We formalize this in the corollary below.
In the next section we will show that for certain types of graph polynomials we can compute the inverse power sums efficiently.

\begin{cor}
\label{cor:approx}
Given $M>0$ and $t \in \C$ satisfying $|t|<M$, there exists a constant $C = C(t,M)\leq (1-|t|/M)^{-1}$ such that the following holds.
Suppose $p$ is a polynomial given by $p(z) = a_0 + a_1z + \cdots + a_dz^d$ with no roots in the open disk $D$ of radius $M$. Suppose further that we are able to compute $a_0$ and the inverse power sums $p_1, \ldots, p_r$ in time $\tau(r)$ for each $r = 1, \ldots, d$.
Then we can compute a multiplicative $\varepsilon$-approximation to $p(t)$ in time $O(\tau(m))$, where $m = C \ln(d / \varepsilon)$. 
\end{cor}
\begin{proof}
The corollary is immediate from (\ref{eq:taylor}), (\ref{eq:TayPowSum}) and Lemma~\ref{lem:approx}.
\end{proof}

\section{Computing coefficients of graph polynomials}\label{sec:coef}
In this section we present our main technical contribution, which is an efficient way to compute the inverse power sums (and hence the coefficients) of a large class of graph polynomials for bounded degree graphs. Throughout, we will focus on graph polynomials whose coefficients can be expressed as linear combinations of induced subgraph counts. Note that in this section we make no assumptions on the locations of the roots of polynomials.
The results in this section are stated only for graphs, but are in fact valid for multigraphs. So the reader could read multigraph instead of graph everywhere in this section. (The \emph{degree} of a vertex in a multigraph is the number of edges incident with the vertex, where a loop is counted twice.)

We start with some definitions after which we state the main result of this section.
Two graphs $H=(V_H,E_H)$ and $G=(V_G,E_G)$ are said to be \emph{isomorphic} if there exists a bijection $f:V_H \rightarrow V_G$ such that for any $u,v \in V_H$, we have that $f(u)f(v) \in E_G$ if and only if $uv \in E_H$. 
We say that $H$ is an induced subgraph of $G$ if there is a subset $S\subseteq V_G$ such that $H$ is isomorphic to $G[S]$, the graph induced by $S$.
We write $\ind(H,G)$ for the number of sets $S\subseteq V_G$ such that $H$ is isomorphic to $G[S]$ (i.e. the number of induced subgraphs of $G$ isomorphic to $H$). 
Note that if $H$ is the empty graph we have $\ind(H,G)=1$ for all $G$.

By $\G$ we denote the collection of all graphs and by $\G_k$ for $k\in \N$ we denote the collection of graphs with at most $k$ vertices. 
A graph invariant is a function $f:\G\to S$ for some set $S$ that takes the same value on isomorphic graphs.
A \emph{graph polynomial} is a graph invariant $p:\G\to \C[z]$, where $\C[z]$ denotes the ring of polynomials in the variable $z$ over the field of complex numbers.
Call a graph invariant $f$ \emph{multiplicative} if $f(\emptyset)=1$ and  $f(G_1\cup G_2)=f(G_1)f(G_2)$ for all graphs $G_1,G_2$ (here $G_1\cup G_2$ denotes the disjoint union of the graphs $G_1$ and $G_2$).
\begin{df}\label{df:bigcp}
Let $p$ be a multiplicative graph polynomial defined by
\begin{equation}\label{eq:mult graph pol}
p(G)(z):=\sum_{i=0}^{d(G)} e_{i}(G)z^{i}
\end{equation}
for each $G\in \G$ with $e_0(G)=1$.
We call $p$ a \emph{bounded induced graph counting polynomial (BIGCP)} if there exists constants $\alpha,\beta \in \N$ such that the following two conditions are satisfied:
\begin{itemize}
\item[(i)] for every graph $G$, the coefficients $e_i$ satisfy
\begin{equation}\label{eq:ind coef}
e_i(G):=\sum_{H\in \G_{\alpha i}}\lambda_{H,i}\ind(H,G)
\end{equation}
for certain $\lambda_{H,i}\in \C$;
\item[(ii)] for each $H\in \G_{\alpha i}$, the coefficients $\lambda_{H,i}$  can be computed in time $O(\beta ^{|V(H)|})$.
\end{itemize}
\end{df}
If, for example, for each $i$, the coefficient $e_i(G)$ in \eqref{eq:mult graph pol} is equal to the number of independent sets of size $i$ in $G$, then it is easy to see that $p$ (which is of course the independence polynomial) is a BIGCP.
In this case the obvious brute force algorithm to compute the coefficient $e_i(G)$ for an $n$-vertex graph $G$ runs in time $O(n^i)$ (by checking all $i$-subsets of $V(G)$) and if $i=O(\ln n)$ then this is quasi-polynomial time. 
Our main result of this section is a general algorithm for computing inverse power sums of BIGCPs (and hence the coefficients of BIGCP's by Proposition~\ref{pr:newton}), which when applied to this example, computes $e_i(G)$ in polynomial time even when $i=O(\ln n)$ as long as the maximum degree of $G$ is bounded.

\begin{theorem}\label{thm:compute coef}
Let $C>0$ and $\Delta\in \N$ and let $p(\cdot)$ be a bounded induced graph counting polynomial.
Then there is a deterministic $(n/\eps)^{O(1)}$-time algorithm, which, 
given any $n$-vertex graph $G$ of maximum degree at most $\Delta$ and any $\eps>0$, computes the inverse power sums $p_1,\ldots,p_m(G)$ of $p(G)$ for $m=C\ln(n/\eps)$.
\end{theorem}

\begin{rem}
The algorithm in the theorem above only has access to the polynomial $p$ via condition (ii) in the definition of BIGCP, that is, it relies only on the algorithm which computes the complex numbers $\lambda_{H,i}$.
\end{rem}
\begin{rem} \label{rem:running time} Assuming $\Delta\geq 3$, the proof of the theorem above in fact yields a running time of $n^{C_1}(n/\eps)^{C_2} = (n/\eps)^{O(1)}$, where $C_1$ can be explicitly determined (and does not depend on $\alpha,\beta,C$ and $\Delta$) and where we can crudely take $C_2=10C\alpha\ln(\beta \Delta)$, where $\alpha$ and $\beta$ are the constants from the definition of BIGCP.
\end{rem}
Before we prove Theorem \ref{thm:compute coef} we will first gather some facts about induced subgraph counts and the number of connected induced subgraphs of fixed size that occur in a graph which we will need for the proof.

\subsection{Induced subgraph counts}
Define $\ind(H,\cdot):\G\to \C$ by $G\mapsto \ind(H,G)$. 
So we view $\ind(H,\cdot)$ as a graph invariant.
We can take linear combinations and products of these invariants.
In particular, for two graphs $H_1,H_2$ we have
\begin{equation}\label{eq:product of ind}
\ind(H_1,\cdot)\cdot \ind(H_2,\cdot)=\sum_{H\in \G}c^H_{H_1,H_2}\ind(H,\cdot),
\end{equation}
where for a graph $H$, $c^H_{H_1,H_2}$ is the number of ordered pairs of subsets of $V(H)$, $(S,T)$, such that $S \cup T = V(H)$ and $H[S]$ is isomorphic to $H_1$ and $H[T]$ is isomorphic to $H_2$.
In particular, given $H_1$ and $H_2$, $c^H_{H_1,H_2}$ is nonzero for only a finite number of graphs $H$.

Computing the parameter $\ind(H,G)$ is generally difficult, but it becomes easier if $H$ is connected (and $V(H)$ is not too large) and $G$ has bounded degree.

\begin{lemma}\label{lem:connected count}
Let $H$ be a connected graph on $k$ vertices and let $\Delta\in \N$. Then
\begin{itemize}
\item[(i)] there is an $O(n\Delta^{k-1})$-time algorithm, which, given any $n$-vertex graph $G$ with maximum degree at most $\Delta$, checks whether $\ind(H,G)\neq 0$;
\item[(ii)] there is an $O(k^2n^{2}\Delta^{2(k-1)})$-time algorithm, which, given any $n$-vertex graph $G$ with maximum degree at most $\Delta$, computes the number $\ind(H,G)$.
\end{itemize}
\end{lemma}
Note that Lemma~\ref{lem:connected count} (i) enables us to test for graph isomorphism between bounded degree graphs when $|V(G)| = |V(H)|$.
\begin{proof}
Let us list the vertices of $V(H)$, $v_1,\ldots,v_k$ in such a way that for $i\geq 1$ vertex $v_i$ has a neighbour among $v_1,\ldots,v_{i-1}$.
Then to embed $H$ into $G$ we first select a target vertex for $v_1$ and then given that we have embedded $v_1, \ldots, v_{i-1}$ with $i\geq 2$ there are at most $\Delta$ choices for where to embed $v_i$. After $k$ iterations, we have a total of at most $n\Delta^{k-1}$ potential ways to embed $H$ and each possibility is checked in the procedure above. Hence we determine if $\ind(H,G)$ is zero or not in $O(n\Delta^{k-1})$ time.

The procedure above gives a list $L$ (of size at most $n\Delta^{k-1}$) of all sets $S \subseteq V(G)$ such that $G[S]$ is isomorphic to $H$, although the list may contain repetitions. It takes time $O(k^2|L|^2) = O(k^2n^2\Delta^{2(k-1)})$ to eliminate repetitions (by comparing every pair of elements in $L$), and the length of the resulting list gives the value of $\ind(H,G)$.
\end{proof}

Next we consider how to enumerate all possible connected induced subgraphs of fixed size in a bounded degree graph.
We will need the following result of Borgs, Chayes, Kahn, and Lov\'asz \cite[Lemma 2.1]{BCKL12}:
\begin{lemma}\label{lem:graph count}
Let $G$ be a graph of maximum degree $\Delta$. Fix a vertex $v_0$ of $G$.
Then the number of connected induced subgraphs of $G$ with $k$ vertices containing the vertex $v_0$ is at most $\frac{(e\Delta)^{k-1}}{2}$.
\end{lemma}

As a consequence we can efficiently enumerate all connected induced subgraphs of logarrithmic size that occur in a bounded degree
graph $G$.
\begin{lemma}\label{lem:enumerate}
There is a $O(n^2k^7 (e\Delta)^{2k})$-time algorithm which, given $k \in \mathbb{N}$ and an $n$-vertex graph $G=(V,E)$ of maximum degree $\Delta$, outputs $\mathcal{T}_k$, the set of all $S \subseteq V$ satisfying $|S| \leq k$ and $G[S]$ connected.
\end{lemma}
\begin{proof}
By the previous result, we know that $|\mathcal{T}_k| \leq nk(e \Delta)^{k-1}$ for all $k$.

We inductively construct $\mathcal{T}_k$. For $k=1$, $\mathcal{T}_k$ is clearly the set of singleton vertices and takes time $O(n)$ to output.

Given that we have found $\mathcal{T}_{k-1}$ we compute $\mathcal{T}_k$ as follows. We first compute the multiset
\[
\mathcal{T}_k^* 
= \{ S \cup \{v\}: S \in \mathcal{T}_{k-1} 
\text{ and } v \in N_G(S) \}.
\]
Here $|N_G(S)| \leq |S|\Delta \leq k\Delta$ and takes time 
$O(k\Delta)$ to find (assuming $G$ is given in adjacency list form). Therefore computing $\mathcal{T}_k^*$ takes
 time $O(|\mathcal{T}_{k-1}|k\Delta) = O(nk^2(e\Delta)^{k})$, which is also the size of $\mathcal{T}_k^*$. Finally we compute the set $\mathcal{T}_k$ by removing the repetitions in $\mathcal{T}_k^*$ (by comparing each element with all previous elements), which takes time $O(k^2|\mathcal{T}_k^*|^2) = O(n^2k^6(e\Delta)^{2k})$.
 
Starting from $\mathcal{T}_1$, we perform the above iteration $k$ times, requiring a total running time of $O(n^2k^7(e\Delta)^{2k})$.
 
 It remains only to show that $\mathcal{T}_k$ contains all 
 the sets we desire. Clearly $\mathcal{T}_{k-1} \subset \mathcal{T}_k$ and assume by induction that $\mathcal{T}_{k-1}$ contains all $T \subset V$ of size $k-1$ with $G[T]$ connected. 
 Given $S\subseteq V$ such that $|S|=k$ and 
 $G[S]$ is connected, take any tree of $G[S]$, remove 
 a leaf $v$ and call the resulting set of vertices $S'$. 
 Then it is clear that $S' \in \mathcal{T}_{k-1}$ and this 
 implies $S  = S' \cup \{v\} \in \mathcal{T}_k$.
\end{proof}

We call a graph invariant $f:\G\to \C$ \emph{additive} if for each $G_1,G_2\in \G$ we have  $f(G_1\cup G_2)=f(G_1)+f(G_2)$.
The following lemma is a variation of a lemma due to Csikv\'ari and Frenkel \cite{CF12}; it is fundamental to our approach.
\begin{lemma}\label{lem:additivity}
Let $f:\G\to \C$ be a graph invariant given by $f(\cdot):=\sum_{H\in \G}a_H\ind(H,\cdot)$. 
Then $f$ is additive if and only if $a_H=0$ for all graphs $H$ that are disconnected. 
\end{lemma}
\begin{proof}
Let $H$ be connected.
Then for $G_1,G_2\in \G$ we have $\ind(H,G_1\cup G_2)=\ind(H,G_1)+\ind(H,G_2)$, as $H$ is connected. Thus $\ind(H,\cdot)$ is additive.
Clearly, linear combinations of additive graph parameters are again additive.
This implies that if $f$ is supported on connected graphs, then $f$ is additive.

Suppose next that $f$ is additive. We need to show that $a_H=0$ if $H$ is disconnected.
By the previous part of the proof, we may assume that $a_H=0$ for all connected graphs $H$.
Let now $H=H_1\cup H_2$ with both $H_1$ and $H_2$ nonempty.
We may assume by induction that for all graphs $H'$ of order strictly smaller than $k:=|V(H)|$ we have $a_{H'}=0$.
Now, by additivity we have
\[
f(H)=f(H_1)+f(H_2)=\sum_{H': |V(H')|\geq k}a_{H'}(\ind (H',H_1)+\ind(H',H_2))=0,
\]
since $|V(H_i)|<k$ for $i=1,2$. On the other hand we have 
\[
f(H)=\sum_{H': |V(H')|\geq k}a_{H'}\ind (H',H)=a_H\ind(H,H).
\]
As $\ind(H,H)\neq 0$, this implies that $a_H=0$ and finishes the proof.
\end{proof}

\subsection{Proof of Theorem \ref{thm:compute coef}}

Recall that $p(\cdot)$ is a bounded induced graph counting polynomial (BIGCP). Given an $n$-vertex graph $G$ with maximum degree at most $\Delta$, we must show how to compute the first $m$ inverse power sums $p_1, \ldots, p_m$ of $p(G)$ in time $(n/ \varepsilon)^{O(1)}$, where $m = C\ln(n / \varepsilon)$. 
To reduce notation, let us write $p=p(G)$, $d=d(G)$ for the degree of $p$, and $e_i=e_i(G)$ for $i=0,\ldots,d$ for the  coefficients of $p$ (from (\ref{eq:mult graph pol})).  
Recall that $p_k:= \zeta_1^{-k} + \cdots + \zeta_d^{-k}$, where $\zeta_1,\ldots,\zeta_d\in \C$ are the roots of the polynomial $p(G)$. 

Noting $e_0=1$, Proposition~\ref{pr:newton} gives
\begin{equation}\label{eq:newton}
p_k= -ke_k - \sum_{i=1}^{k-1}e_ip_{k-i},
\end{equation}
for each $k=1, \ldots, d$. 
By \eqref{eq:ind coef}, for $i\geq 1$, the $e_i$ can be expressed as linear combinations of induced subgraph counts of graphs with at most $\alpha i$ vertices. Since $p_1=-e_1$, this implies that the same holds for $p_1$.
By induction, \eqref{eq:product of ind}, and \eqref{eq:newton} we have that for each $k$
\begin{equation}\label{eq:power to ind}
p_k=\sum_{H\in \G_{\alpha k}}a_{H,k}\ind(H,G),
\end{equation}
for certain, yet unknown, coefficients $a_{H,k}$.

Since $p$ is multiplicative, the inverse power sums are additive. Thus Lemma \ref{lem:additivity} implies that $a_{H,k}=0$ if $H$ is not connected. 
Denote by $\mathcal{C}_{i}(G)$ the set of connected graphs of order at most $i$ that occur as induced subgraphs in $G$.
This way we can rewrite \eqref{eq:power to ind} as follows:
\begin{equation}\label{eq:power support}
p_k=\sum_{H\in \mathcal{C}_{\alpha k}(G)}a_{H,k}\ind(H,G).
\end{equation}
The next lemma says that we can compute the coefficients $a_{H,k}$ efficiently for $k=1,\ldots,m$, where $m=C\ln(n/ \eps)$.
\begin{lemma}\label{lemma:compute}
There is an $(n/\eps)^{O(1)}$-time algorithm, which given a BIGCP $p$ and an $n$-vertex graph $G$ of bounded maximum degree, computes and lists the coefficients $a_{H,k}$ in \eqref{eq:power support} for all $H\in \mathcal{C}_{\alpha k}(G)$ and all $k=1,\ldots,m = C \ln(n / \varepsilon)$.
\end{lemma}
\begin{proof}
Using the algorithm of Lemma \ref{lem:enumerate}, we first compute the sets $\mathcal{T}_{\alpha k}$ consisting of all subsets $S$ of $V(G)$ such that $|S|\leq \alpha k$ and $G[S]$ is connected, for $k=1\ldots, m$.
This takes time bounded by $(n/\eps)^{O(1)}$.
(Note that the algorithm in Lemma \ref{lem:enumerate} computes and lists all the sets $\mathcal{T}_{i}$ for $i=1,\ldots,\alpha m$.)
We next compute and list the graphs in $\mathcal{C}_{\alpha k}(G)$
by considering the set of graphs $\{G[S] \mid S \in \mathcal{T}_{\alpha k} \}$ and removing copies of isomorphic graphs
using Lemma \ref{lem:connected count} (i) to test for isomorphism.
This takes time at most $(n/\eps)^{O(1)}$
 for each $k$,
so the total time to compute and list the $\mathcal{C}_{\alpha k}(G)$ is bounded by $(n/\eps)^{O(1)}$.

To prove the lemma, let us fix $k\leq m$ and show how to compute the coefficients $a_{H,k}$, assuming that  we have already computed and listed the coefficients $a_{H,k'}$ for all $k'<k$.
Let us fix $H\in \mathcal{C}_{\alpha k}(G)$.
By (\ref{eq:newton}), it suffices to compute the coefficient of $\ind(H,\cdot)$ in $p_{k-i}e_{i}$ for $i=1,\ldots,k$ (where we set $p_0=1)$.
By \eqref{eq:ind coef}, \eqref{eq:product of ind} and \eqref{eq:power to ind} we know that the coefficient of $\ind(H,\cdot)$ in $p_{k-i}e_{i}$ is given by
\begin{equation}\label{eq:coef e_{k-i}p_i}
\sum_{H_1,H_2} c^H_{H_1,H_2}a_{H_2,(k-i)}\lambda_{H_1,i}=\sum_{(S,T) : S \cup T = V(H)} a_{H[T], (k-i)}\lambda_{H[S], i} .
\end{equation}
As $|V(H)|\leq \alpha k=O(\ln(n/\eps))$, the second sum in \eqref{eq:coef e_{k-i}p_i} is over at most $4^{\alpha k}=(n/\eps)^{O(1)}$ pairs $(S,T)$.
For each such pair, we need  to compute $\lambda_{H[S], i}$ and look up $ a_{H[T], (k-i)}$.
We can compute $\lambda_{H[S],i}$ in time bounded by $\beta ^{|S|}=(n/\eps)^{O(1)}$ since $p$ is a BIGCP.


Looking up $a_{H[T],(k-i)}$ in the given list requires us to test isomorphism of $H[T]$ with each graph in $\mathcal{C}_{\alpha(k-i)}(G)$ (noting that $a_{H[T],(k-i)}=0$ if $H[T] \not \in \mathcal{C}_{\alpha(k-i)}(G)$ by Lemma~ \ref{lem:additivity}). Using Lemma~\ref{lem:connected count}(i) to test for graph isomorphism, this takes time at most 
\[O(|\mathcal{C}_{\alpha(k-i)}(G)|\alpha(k-i)\Delta^{\alpha(k-i)-1})=O(n/\eps)^{O(1)}.
\]
Here we use Lemma~\ref{lem:graph count} to bound $|\mathcal{C}_{\alpha(k-i)}(G)|$.
Together, all this implies that the coefficient of $\ind(H,\cdot)$ in $p_{k-i}e_{i}$ can be computed in time bounded by $(n/\eps)^{O(1)}$, and so the coefficient $a_{H,k}$ can be computed in time $(n/\eps)^{O(1)}$.
Thus all coefficients $a_{H,k}$ for $H\in \mathcal{C}_{\alpha k}(G)$ can be computed and listed in time bounded by  $|\mathcal{C}_{\alpha k}(G)|(n/\eps)^{O(1)}=(n/\eps)^{O(1)}$.
This can be done for each $k = 1, \ldots, m$ in time $(n/\eps)^{O(1)}$.
\end{proof}
To finish the proof of the theorem, we compute $p_k$ for each $k=1,\ldots, m$ by adding all the numbers $a_{H,k}\ind(H,G)$ over all $H\in \mathcal{C}_{\alpha k}(G)$. This can be done in time 
\[O(m|\mathcal{C}_{\alpha m}(G)|(\alpha m)^2n^2\Delta^{2(\alpha m-1)})=(n/\eps)^{O(1)},\] where we use that computing $\ind(H,G)$ with $H\in \mathcal{C}_{\alpha k}(G)$ takes time $O((\alpha k)^2n^2\Delta^{2(\alpha k-1)})$ by Lemma~\ref{lem:connected count}(ii).

\subsection{Extensions to colored graphs}\label{subsec:color}

In this section, we treat the case of colored graphs, which we shall require later. In fact all earlier proofs go through line by line for the colored case, but we chose to avoid excessive generality for the sake of exposition. Here we restate various definitions for the colored case and restate the main theorem.

Let us fix $\N$, the natural numbers, as a set of colors.
Let $H=(V_H,E_H)$ and $G=(V_G,E_G)$ be graphs whose vertices (resp.\  edges)  are colored from the set $T$ (not necessarily properly).
For vertex-colored graphs $H=(V_H,E_H)$ and $G=(V_G,E_G)$ we define $H$ and $G$ to be \emph{color-isomorphic} if there is a color isomorphism 
between $H$ and $G$, that is, a graph isomorphism $f: V_H \rightarrow 
V_G$ such that $u$ and $f(u)$ have the same color for every $u \in 
V(H)$. For edge-colored graphs $H=(V_H,E_H)$ and $G=(V_G,E_G)$ we define $H$ and $G$ to be \emph{color-isomorphic} if there is a color isomorphism 
between $H$ and $G$, that is, a graph isomorphism $f: V_H \rightarrow 
V_G$ such that $uv$ and $f(uv)$ have the same color for every edge $uv \in 
V(H)$.
In both the vertex and edge colored cases, we say that $H$ is an induced colored-subgraph of $G$ if there is a subset $S\subseteq V_G$ such that $H$ is color-isomorphic to $G[S]$, the graph induced by $S$ (which inherits colors from the coloring of $G$).
We write $\ind_{vc}(H,G)$ 
(resp.\ $\ind_{ec}(H,G)$) 
for the number of sets $S\subseteq V_G$  such that $H$ is color-isomorphic to $G[S]$ (i.e. the number of induced colored-subgraphs of $G$ isomorphic to $H$). 

We can then extend the definitions in the natural way to their colored versions. In particular $\G$, $\G_k$, $\mathcal{C}_k(G)$ become respectively the set of all vertex (resp.\ edge) colored graphs, the set of all vertex (resp.\ edge) colored graphs of order at most $k$, and the set of all connected, vertex (resp.\ edge) colored induced subgraphs of $G$ of order at most $k$. Note that $\G_k$ becomes infinite in the colored setting but is finite in the uncolored setting, but this will not matter. The definition of (multiplicative) graph invariant and graph polynomial extend in the natural way to vertex (resp.\ edge) colored graphs and in particular a BIGCP for vertex (resp.\ edge) colored graphs is defined in exactly the same way. We need only note that although the sum in part (i) of the definition of BIGCP is infinite in the colored version, all but finitely many terms will be zero when evaluating for a particular choice of colored graph $G$. Now the colored version of Theorem~\ref{thm:compute coef} reads as follows.

\begin{theorem}\label{thm:compute coef colored}
Let $C>0$ and $\Delta\in \N$ and let $p(\cdot)$ be a bounded induced graph counting polynomial for vertex (resp.\ edge) colored graphs.
Then there is a deterministic $(n/\eps)^{O(1)}$-time algorithm, which, 
given any $n$-vertex vertex (resp.\ edge) colored graph $G$ of maximum degree at most $\Delta$ and any $\eps>0$, computes the inverse power sums $p_1,\ldots,p_m(G)$ of $p(G)$ for $m=C\ln(n/\eps)$.
\end{theorem}

\section{The independence polynomial}\label{sec:independence}
\subsection{The independence polynomial on bounded degree graphs}
\begin{proof}[Proof of Theorem \ref{thm:ind general}]
First note that by a result of Shearer \cite{S98}, cf. Scott and Sokal \cite[Corollary 5.7]{SS05}, we know that $Z(G)(z)\neq 0$ for all all graphs $G$ of maximum degree at most $\Delta$ and all $z\in \C$ that satisfy $|z|\leq \lambda^*=\frac{(\Delta-1)^{\Delta-1}}{\Delta^{\Delta}}$. 

Now fix an $n$-vertex graph $G$ of maximum degree at most $\Delta$.
Let $m = C\ln(n / \varepsilon)$, where $C = C(\lambda, \lambda^*)$ is the constant in Corollary~\ref{cor:approx}.
As the $k$th coefficient of $Z(G)$ is equal to $\ind(\bullet^k,G)$, where $\bullet^k$ denotes the graph consisting of $k$ isolated vertices and as $Z(G)$ is clearly multiplicative and has constant term equal to $1$, we have that $Z(G)$ is a BIGCP (taking $\alpha=\beta =1$).
So by Theorem \ref{thm:compute coef} we see that for $k=1,\ldots,m$ we can compute the first $m$ inverse power sums of $Z(G)$ in time $(n/ \varepsilon)^{O(1)}$.
Noting that the degree of $Z(G)$ is at most $n$, Corollary~\ref{cor:approx} implies we can compute a multiplicative $\varepsilon$-approximation to $Z(G)(\lambda)$ in time $(n/ \varepsilon)^{O(1)}$.
This concludes the proof.
\end{proof}

Evaluating the independence polynomial at negative and complex values
gives us new information about the distribution of independent sets in a graph, as illustrated by the following example.
We denote by $Z_e(G)(\lambda)$ the polynomial defined in the same way as the independence polynomial except that in the sum \eqref{eq:def ind pol}, we only allow independent sets whose cardinality is even.

\begin{theorem}
Let $\Delta\in \N$ and let $0 \leq \lambda <\lambda^*(\Delta):=\frac{(\Delta-1)^{\Delta-1}}{\Delta^{\Delta}}$. Then there exists a deterministic algorithm, which, given a graph $G=(V,E)$ of maximum degree at most $\Delta$ and $\eps>0$, computes a multiplicative $\eps$-approximation to $Z_e(G)(\lambda)$ in time $(|V|/\eps)^{O(1)}$.
\end{theorem}
\begin{proof}
We apply the algorithm of Theorem~\ref{thm:ind general} to compute  multiplicative $\varepsilon$-approximations $A(\lambda)$ and $A(- \lambda)$ to $Z(G)(\lambda)$ and $Z(G)(-\lambda)$ respectively in time $(|V|/\varepsilon)^{O(1)}$. We have
\[
e^{-\varepsilon}Z(G)(\lambda) \leq A(\lambda) \leq  e^{\varepsilon}Z(G)(\lambda)
\:\:\: \text{ and } \:\:\:
e^{-\varepsilon}Z(G)(-\lambda) \leq A(-\lambda) \leq  e^{\varepsilon}Z(G)(- \lambda).
\]
Taking half the sum of these equations and noting that $Z_e(G)(\lambda) = \frac{1}{2}(Z(G)(\lambda) + Z_G(- \lambda))$, we see that $\frac{1}{2}(A(\lambda) + A(- \lambda))$ is a multiplicative $\varepsilon$-approximation to $Z_e(G)(\lambda)$ provided both $Z(G)(\lambda)$ and $Z(G)(-\lambda)$ have the same sign. 

Clearly $Z(G)(\lambda)>0$ since the coefficients of $Z(G)$ are nonnegative real numbers. Also $Z(G)(-\lambda)>0$ because we know by the result of Scott and Sokal \cite{SS05} that $Z(G)$ does not vanish in the interval $[-\lambda^*, \lambda^*]$, and we know $Z(G)$ is positive in the interval $[0, \lambda^*]$ since all the coefficients of $Z(G)$ are nonnegative real numbers. Hence $Z(G)$ is positive on the whole interval $[-\lambda^*, \lambda^*]$ and in particular $Z(G)(\lambda) >0$.
\end{proof}

\subsection{The multivariate independence polynomial} \label{subsec:mult}
Here we will briefly mention how our results apply to the multivariate independence polynomial.
For a graph $G=(V,E)$ and a variable $z_v$ for each $v\in V$ define 
\[
Z(G)((z_v)_{v\in V})=\sum_{\substack{I\subseteq V\\\text{independent}}}\prod_{v\in I}z_v.
\]
Fix the complex values of $z_v$, $v \in V$ at which we wish to evaluate the multivariate independence polynomial.
Now define a univariate graph polynomial $q(G)$ by $q(G)(\lambda)=Z(G)((\lambda z_v)_{v\in V})$, where $q(G)(1)$ is what we wish to estimate.
The coefficient of $\lambda^k$ in $q(G)$ is then given by the sum over all independent sets of $G$ of size $k$, where an independent set $I$ is counted with weight $\prod_{v\in I}z_v$.
While we can no longer represent this as a linear combination of ordinary induced subgraph counts, we can view this as a linear combination of vertex-colored induced subgraph counts, as we will now explain.

We extend $q$ to vertex-colored graphs and denote this extension by $q_{vc}$. 
Associate a variable $z_i$ to each color $i=1,2,\ldots$.
The coefficient of $\lambda^k$ of the polynomial $q_{vc}$ is expressed as $\sum_{H}\alpha_{H} \ind_{vc}(H,\cdot)$, where the sum is over all vertex-colored graphs $H=(V_H,E_H)$ with coloring $c:V_H\to \N$ on $k$ vertices and where $\alpha_{H}=\prod_{u\in V_H}z_{c(u)}$ if $H$ is an independent set and zero otherwise.
Then $q_{vc}$ is a BIGCP in the vertex-colored setting, cf. Subsection~\ref{subsec:color}. 

Suppose now that $G$ has $n$ vertices labelled $1,\ldots,n$. 
We view $G$ as a vertex-colored graph by giving the vertex labeled $i$ color $i$ for $i=1,\ldots,n$. 
Then $q_{vc}(G)=q(G)$.
So if $G$ has bounded degree, we can by Theorem~\ref{thm:compute coef colored} compute the logarithmic order inverse power sums of $q(G)$ efficiently, as in the previous section.

 The result of Shearer \cite{S98}, cf. Scott and Sokal \cite[Corollary 5.7]{SS05} also applies to the multivariate independence polynomial (i.e.\ $Z(G)((z_v)_{v\in V})$ is non-zero whenever $|z_v|<\lambda^*(\Delta)$ for all $v \in V$ and $\Delta(G) \leq \Delta$). This means $q(G)(\lambda)$ is nonzero whenever $|\lambda| < M$ for some $M>1$. It then follows that we can efficiently approximate $q(G)(1)$ and hence $Z(G)((z_v)_{v\in V})$ if $G$ has maximum degree at most $\Delta $ and if all $z_v$ satisfy $|z_v|<\lambda^*(\Delta)$.

\subsection{The independence polynomial on claw-free graphs}\label{subsec:clawfree}

In this subsection, we illustrate a technique of Barvinok for approximating graph polynomials on larger regions of the complex plane by making careful polynomial transformations. We use this technique to prove Theorem~\ref{thm:ind claw free}, which shows that we can approximate the independence polynomial of claw-free graphs on almost the entire complex plane. First we require a few preliminary results.
\begin{proposition}
\label{pr:roots}
If $G$ is a claw-free graph of maximum degree $\Delta$ and $\zeta$ is a root of the independence polynomial $Z(G)$ of $G$ then $\zeta \in \mathbb{R}$ with $\zeta <  -\frac{1}{e(\Delta  -1)}$.
\end{proposition}
\begin{proof}
The fact that $\zeta \in \mathbb{R}$ is a result of Chudnovsky and Seymour \cite{CS7}. The fact that $\zeta$ must be negative follows  because all the coefficients of $Z(G)$ are positive. Now the result of Shearer \cite{S98} states that 
\[
|\zeta| \geq \lambda^*(\Delta) = \frac{(\Delta-1)^{\Delta - 1}}{\Delta^{\Delta}} > \frac{1}{e(\Delta-1)},
\]  
from which the proposition follows.
\end{proof}

We also require the following lemma of Barvinok~\cite{B16}.
\begin{lemma} \label{lem:transform}
For $\rho \in (0,1)$ we define
\begin{align*}
&\alpha = \alpha(\rho) = 1 - \exp( - \rho^{-1}),
&&\beta = \beta(\rho) = \frac{ 1 - \exp( -1 - \rho^{-1}) }{ 1 - \exp( - \rho^{-1}) } > 1, \\
&N = N(\rho) = \lfloor (1 + \rho^{-1})\exp(1 + \rho^{-1}) \rfloor,
&&\sigma = \sigma(\rho) = \sum_{i=1}^N \frac{\alpha^i}{i}.
\end{align*}
The polynomial
\[
\phi(z) = \phi_{\rho}(z) = \frac{1}{\sigma}\sum_{i=1}^N \frac{(\alpha z)^i}{i}
\]
has the following properties:
\begin{itemize}
\item[(i)] $\phi(0) = 0$ and $\phi(1) = 1$ and $\phi$ has degree $N$;
\item [(ii)] If $z \in \mathbb{C}$ with $|z| \leq \beta$ then $\phi_{\rho}(z) \in S_{\rho}$, where
\[
S_{\rho} := \{z \in \mathbb{C} \mid
-\rho \leq \Re(z) \leq 1+ 2\rho \:\:\:\text{ and }\:\:\:
-2\rho \leq \Im(z) \leq 2\rho
\}.
\] 
\end{itemize}
\end{lemma}

\begin{proposition}
\label{pr:rotation}
Fix $\lambda = re^{i \theta} \in \mathbb{C}$ with $\theta \in (-\pi, \pi)$. Let $S_{\rho}$ be as in the previous lemma, and let $\mathbb{R}^-$ denote the negative real line. Then 
\begin{align*}
\lambda S_{\rho} \cap \mathbb{R}^- \subseteq 
\begin{cases}
[-\sqrt{5}\rho r, 0] &\text{ if } \theta \in [-\frac{\pi}{2}, \frac{\pi}{2}]; \\
[-2\rho r/|\sin \theta|, 0] &\text{ otherwise}.
\end{cases}
\end{align*}
\end{proposition}
\begin{proof}
$S_{\rho}$ is a bounded strip parallel to the real axis in the complex plane, so $\lambda S_{\rho}$ is the same strip enlarged by a factor $r$ and rotated by an angle $\theta$. The proposition then follows from elementary trigonometry. 
\end{proof}

\begin{proof}[Proof of Theorem \ref{thm:ind claw free}]
Recall that we are given a claw-free graph $G$ of maximum degree $\Delta$ and $\lambda \in \mathbb{C}$ that is not a negative real number and we wish to find a multiplicative $\varepsilon$-approximation to $Z(G)(\lambda)$. 

Set $n:=|V(G)|$ and let $\lambda = re^{i \theta}$ with $\theta \in (-\pi, \pi)$. Set 
\begin{align*}
\rho =
\begin{cases}
  1 / 9r(\Delta - 1) &\text{ if } \theta \in [-\frac{\pi}{2}, \frac{\pi}{2}]; \\
 |\sin \theta| / 6r(\Delta - 1) &\text{ otherwise},
\end{cases}
\end{align*}
and consider the polynomial $g(z) = Z(G)(\lambda \phi_{\rho}(z))$. Note that the degree of $g$ is $O(n)$ since the degree of $Z(G)$ is at most $n$ and the degree of $\phi_{\rho}$ is a constant $N(\rho)$.

We will use Corollary~\ref{cor:approx} to find a multiplicative $\varepsilon$-approximation to $g(1) = Z(G)(\lambda)$ in time $(n/ \varepsilon)^{O(1)}$.
In order to apply Corollary~\ref{cor:approx} to draw this conclusion, it is enough to check that (i) $g$ has no roots in the disk $|z| \leq \beta := \beta(\rho)$ and that (ii) the first $m = C \ln(d / \varepsilon)$ inverse power sums of $g$ can be computed in time $(n/ \varepsilon)^{O(1)}$, where $d = O(n)$ is the degree of $g$ and $C = C(\beta, 1)$ is the constant in the statement of Corollary~\ref{cor:approx}.

It remains to check (i) and (ii). To see (i), note first that by Lemma~\ref{lem:transform}, $\phi_{\rho}$ maps the disk $D= \{z \in \mathbb{C} \mid |z| \leq \beta\}$ into $S_{\rho}$. By Proposition~\ref{pr:rotation} and our choice of $\rho$, we have 
$\lambda \phi_{\rho}(D) \cap \mathbb{R}^- \subseteq  (-\frac{1}{3(\Delta - 1)},0]$. By Proposition~\ref{pr:roots} we know that if $Z(G)(\zeta) = 0$ then $\zeta \in \mathbb{R}$ with $\zeta <  -\frac{1}{e(\Delta  -1)}$. In particular this implies $g(\cdot) = Z(G)(\lambda \phi_{\rho}( \cdot ))$ has no root in the disk $D$.

For (ii), we show that we can compute the first $m$ coefficients of $g$ in time $(n/\eps)^{O(1)}$, which is sufficient by Proposition~\ref{pr:newton}. 
Given a polynomial $p(z) = \sum_{i=0}^d a_iz^i$, write $p_{[m]}(z) := \sum_{i=0}^m a_iz^i$. 
Then we note that $g_{[m]} = (Z(G)   \circ (\lambda\phi))_{[m]} = (Z(G)_{[m]} \circ (\lambda \phi_{[m]}))_{[m]}$, where we crucially use the fact that $\phi$ has no constant term since $\phi(0) = 0$. 
In words, to obtain $g_{[m]}(z)$ we substitute $\lambda \phi_{[m]}(z)$ into $Z(G)_{[m]}(z)$ and keep the first $m$ terms. Thus, in $O(m^3)$-time we can obtain the first $m$ coefficients of $g$ if we know the first $m$ coefficients of $Z(G)$.
As $Z(G)$ is a BIGCP, we can compute its first $m$ inverse power sums in time $(n/\eps)^{O(1)}$ (as in the proof of Theorem~\ref{thm:ind general}), from which we can find its first $m$ coefficients in time $O(m^2)$ by Proposition~\ref{pr:newton}.
This finishes the proof.
\end{proof}

We remark that, for the $(n/\varepsilon)^{O(1)}$ running time in the algorithm above, the $O(1)$ in the exponent depends on $\lambda$ and grows exponentially fast in $r = |\lambda|$. However, this dependence can be brought down to $O(|\lambda|^{1/2})$ by adapting Lemma~\ref{lem:transform} as described by Barvinok \cite{B16PC}.

\section{The Tutte polynomial}\label{sec:Tutte}
Here we give a proof of Theorem \ref{thm:Tutte}.
\begin{proof}[Proof of Theorem \ref{thm:Tutte}]
By a result of Jackson, Procacci and Sokal, cf. \cite[Theorem 1.2]{JPS13} (which is valid for loopless multigraphs) we know that there exists a constant $K>0$ depending on $\Delta$ and $w$ such that for all $q$ with $|q|>K$ we have $Z_T(G)(q,w)\neq 0$ for all graphs $G$ of maximum degree at most $\Delta$.
This is exactly opposite to what we need to apply Corollary \ref{cor:approx}, so let us define the graph polynomial $p_T$ by 
\begin{equation}\label{eq:inverted tutte}
p_T(G)(z):=z^{|V|}Z_T(G)(1/z,w),
\end{equation}
for any graph $G=(V,E)$. Note that $p_T(G)$ has degree $n:=|V|$ and that if $x$ is a multiplicative $\varepsilon$-approximation to $p_T(G)(1/q,w)$, then $q^nx$ is a multiplicative $\varepsilon$-approximation to $Z_T(G)(q,w)$, so it is sufficient to find the former.

We will show that for any $n$-vertex graph $G$ of maximum degree at most $\Delta$, we can compute the first $m$ inverse power sums of $p_T(G)$ in time $(n/ \varepsilon)^{O(1)}$, where $m = C\ln(n / \varepsilon)$ and $C = C(1/q, 1/K)$ is the constant in Corollary~\ref{cor:approx}.
Corollary~\ref{cor:approx} then implies we can compute a multiplicative $\varepsilon$-approximation to $p_T(G)(1/q)$ and hence to $Z_T(G)(q,w)$ in time $(n/ \varepsilon)^{O(1)}$.

We will show that $p_T(G)$ is a BIGCP so that by Theorem \ref{thm:compute coef} we can conclude that we can compute the first $m$  inverse power sums in time $(n/ \varepsilon)^{O(1)}$. 

Since the Tutte polynomial $Z_T(G)(z,w)$ (as a polynomial in $z$) is a monic and multiplicative graph polynomial (of degree $n = |V(G)|$), we know that the constant term of $p_T$ equals $1$ and that $p_T$ is multiplicative.
So it suffices to show conditions (i) and (ii) in Definition~\ref{df:bigcp}.
The coefficient of $z^k$ in $p_T(G)$ equals the coefficient of $z^{n-k}$ in $Z_T(G)(z,w)$ and is by definition equal to the sum over all subsets $A$ of $E$ such that $A$ induces a graph with exactly $n-k$ components, where each subset is counted with weight $w^{|A|}$.
Let us call a component of a graph \emph{nontrivial} if it consists of more than one vertex.
Suppose some subset of the edges $A\subseteq E$ induces $n-k$ components of which $c$ are nontrivial. 
Then we have $n-k-c$ isolated vertices and so the graph $F$, consisting of the union of these nontrivial components, has $n-(n-k-c)=k+c$ vertices and $k(F)=c$ components. Thus we have a correspondence between subsets $A$ of $E$ that induce a graph with exactly $n-k$ components and subgraphs $F$ of $G$ with no isolated vertices satisfying $k(F)=V(F)-k$.  
Therefore, writing $\delta(F)$ for the minimum degree of the subgraph $F$, the coefficient of $z^{n-k}$ in $Z_T(G)$ can be expressed as
\begin{equation}
\label{eq:tuttecoeff}
\sum_{\substack{F \subseteq G \; : \; \delta(F) \geq 1 \\ |k(F)|=|V(F)|-k}}w^{|E(F)|} = 
\sum_{H} \sum_{\substack{F \subseteq H \; : \; \delta(F) \geq 1 \\ V(F) = V(H) \\ |k(F)|=|V(F)|-k}}w^{|E(F)|}\ind(H,G).
\end{equation}
In fact the first sum can be taken over graphs $H$ with at most $2k$ vertices. This is because $V(H) = V(F)$ and 
\[
 |V(F)| = k(F)+k\leq \frac{V(F)|}{2}+k.
\]
as $F$ has no isolated vertices. 

From (\ref{eq:tuttecoeff}), we can compute the coefficient of $\ind(H,G)$ by checking all subsets of $E(H)$ in time $O(2^{E(H)}) = O(2^{\Delta|V(H)|})$.
 This implies that $p_T$ is a BIGCP (taking  $\alpha = 2$ and $\beta  = 2^{\Delta}$).
\end{proof}

\begin{rem}
Csikv\'ari and Frenkel \cite{CF12} introduced \emph{graph polynomials of bounded exponential type} and showed that these polynomials have bounded roots on bounded degree graphs. 
This was utilized in \cite{R15} to give quasi-polynomial-time approximation algorithm for evaluations of these polynomials.
The Tutte polynomial with the second argument fixed is an example of such a polynomial.
We remark here that the proof given above for the Tutte polynomial also easily extends to graph polynomial of bounded exponential type.
So the algorithm in \cite{R15} can be adapted to run in polynomial time on bounded degree graphs.
\end{rem}

\section{Partition functions of spin models}\label{sec:hom}
In this section we will state and prove a generalization of Theorem \ref{thm:part spin} and we will indicate how our method applies to partition functions of graph homomorphisms with multiplicities.
\subsection{Partition functions for edge-colored graphs}
Let $G=(V,E)$ be a graph.
Suppose also that for each $e\in E$ we have a symmetric $k\times k$-matrix $A^e$. 
Let us write $\mathcal{A}=(A^e)_{e\in E}$.
Then we can extend the definition of the partition function of a spin model as follows:
\begin{equation}
p(G)(\mathcal{A})=\sum_{\phi:V\to [k]}\prod_{e=\{u,v\}\in E} A^{e}_{\phi(u),\phi(v)}.\label{eq:pf edgecolor}
\end{equation}
We will refer to $p(G)(\mathcal{A})$ as the \emph{partition function of $\mathcal{A}$}. 
In \cite{GK12} this is called a Markov random field (if the $A^e$ are nonnegative) and in \cite{LY13} this is called a multi spin system.
Clearly, if all $A^e$ are the same, this just reduces to the partition function of a spin model.
We have the following result, which implies Theorem \ref{thm:part spin}.
\begin{theorem}\label{thm:edge color spin}
Let $\Delta,k\in \N$. Then there exists a deterministic algorithm, which, given a graph $G=(V,E)$ of maximum degree at most $\Delta$, symmetric $k\times k$ matrices $A^e$, $e\in E$, such that $|A^e_{i,j}-1|\leq 0.34/\Delta$ for all $i,j=1,\ldots,k$ and $e\in E$, and $\eps>0$, computes a multiplicative $\eps$-approximation to $p(G)(\mathcal{A})$ in time $(|V|/\eps)^{O(1)}$.
\end{theorem}
\begin{rem}
The constant $0.34$ may be replaced by $0.45$ if $\Delta\geq 3$ and by $0.54$ if $\Delta$ is large enough; 
see \cite{BS14a}.
\end{rem}
\begin{proof}
Let $J$ be the all ones matrix.
For $z\in \C$, let $B^e(z):=J+z(A^e-J)$ and let $\mathcal{A}'(z):=(B^e(z))_{e\in E}$.
Define a univariate polynomial $q$ by
\begin{equation}\label{eq:spin pol}
q(G)(z)=k^{-|V|}p(G)(\mathcal{A}'(z)). 
\end{equation}
Then $q(G)(0)=1$ and $q(G)(1)=k^{-|V|}p(G)(\mathcal{A})$. 
Barvinok and Sober\'on \cite[Theorem 1.6]{BS14a} showed that there exists a constant $\delta>0$ such that $q(G)(z)\neq 0$ for all $z$ satisfying $|z|\leq 1+\delta$.

We will show that for any $n$-vertex graph $G$ of maximum degree at most $\Delta$, we can compute the first $m$ inverse power sums of $q(G)$ in time $(n/ \varepsilon)^{O(1)}$, where $m = C\ln(n / \varepsilon)$ and $C = C(1, 1+ \delta)$ is the constant in Corollary~\ref{cor:approx}.
Noting that the degree of $q(G)$ is at most $|E|\leq n \Delta/2$, Corollary~\ref{cor:approx} implies we can compute a multiplicative $\varepsilon$-approximation to $q(G)(1)$ in time $(n/ \varepsilon)^{O(1)}$. 
So it remains to show that we can compute the first $m$ inverse power sums of $q(G)$ in time $(n/\eps)^{O(1)}$. 

We will show that we can interpret $q$ as an edge-colored BIGCP, cf. Subsection~\ref{subsec:color}.
Suppose $G$ has $\ell$ edges, labeled $1,\ldots,\ell$. Color the edge labeled $i$ with color $i$ for $i=1,\ldots,\ell$.
By definition, $q(G)(z)$ satisfies
\begin{align}
q(G)(z) &= k^{-n}\sum_{\phi:V\to[k]}\prod_{e=\{u,v\}\in E}(J+(z(A^{e}-J)))_{\phi(u),\phi(v)}\nonumber
\\
&=k^{-n}\sum_{i=0}^{|E|} z^i\bigg( \sum_{\substack{F\subseteq E\\|F|=i}}\sum_{\phi:V\to[k]}\prod_{e=\{u,v\}\in F}(A^{e}-J)_{\phi(u),\phi(v)}\bigg).	\label{eq:pf G(F)}
\end{align} 

For a subset $F$ of $E$, define $G[F]$ to be the edge-colored graph induced by the edges in $F$. 
The vertex set of $G[F]$ consists of those vertices incident with edges in $F$ and hence has size at most $2|F|$.
Then we see that the coefficient of $z^i$ in \eqref{eq:pf G(F)} can be written as follows:
\begin{align}
k^{-n}&\sum_{\substack{H\in \G_{2i}\\ |E(H)|=i}}k^{n-|V(H)|}\bigg( \sum_{\phi:V(H)\to[k]}\prod_{e=\{u,v\}\in E(H)}(A^{e}-J)_{\phi(u),\phi(v)}\bigg)\ind_{ec}(H,G)\nonumber
\\
=&\sum_{\substack{H\in \G_{2i}\\ |E(H)|=i}}\bigg(k^{-|V(H)|} \sum_{\phi:V(H)\to[k]}\prod_{e=\{u,v\}\in E(H)}(A^{e}-J)_{\phi(u),\phi(v)}\bigg)\ind_{ec}(H,G),	\label{eq:spin ind}
\end{align}
where we interpret $\G_{2i}$ as the collection of edge-colored graphs on at most $2i$ vertices.
This shows how to extend $q$ to edge-colored graphs.
The inner sum in \eqref{eq:spin ind} can be computed in time $O(k^{|V(H)|})$, and clearly $q$ has constant term equal to $1$ and it is multiplicative.
This implies that the extension of $q$ to edge-colored graphs is a BIGCP (with constant $\alpha=2$ and $\beta =k$) and so Theorem \ref{thm:compute coef colored} implies that we can compute the first $m$ inverse power sums of $q$ in time bounded by $O(n/\eps)^{O(1)}$.
This finishes the proof.
\end{proof}

\subsection{Partition functions of graph homomorphisms with multiplicities}
Let $G=(V,E)$ be graph. 
Suppose that for each $e\in E$ we have a symmetric $k\times k$-matrix $A^e$. 
Let us write $\mathcal{A}=(A^e)_{e\in E}$.
Let $n=|V|$ and let $\mu=(\mu_1,\ldots,\mu_k)$ with $\mu_i\in \Z_{\geq 1}$ for each $i$ be such that that $\sum_{i=1}^k \mu_i=n$.
We call such $\mu$ a \emph{composition} of $n$ in $k$ parts.
Barvinok and Sober\'on \cite{BS14b} define the \emph{partition function of graph homomorphisms with multiplicities $\mu$} as
\begin{equation}
p_\mu(G)(\mathcal{A})=\sum_{\substack {\phi:V\to [k]\\ |\phi^{-1}(i)|=\mu_i}}\prod_{e=\{u,v\}\in E}A^{e}_{\phi(u),\phi(v)}.
\label{eq:part mult}
\end{equation}
We refer to \cite{BS14b} for more details and background on this type of partition function.

Building on a result from Barvinok and Sober\'on \cite[Section 2]{BS14b} and using exactly the same proof as above we directly establish the following:
\begin{theorem}
Let $\Delta,k\in \N$. Then there exists a deterministic algorithm, which, given a graph $G=(V,E)$ of maximum degree at most $\Delta$, a composition $\mu$ of $|V|$ in $k$ parts, symmetric $k\times k$ matrices $A^e$, $e\in E$ such that $|A^e_{i,j}-1|\leq 0.1/\Delta$ for all $i,j=1,\ldots,k$ and $e\in E$, and $\eps>0$, computes an $\eps$-approximation to $p_\mu(G)(\mathcal{A})$ in time $(|V|/\eps)^{O(1)}$.
\end{theorem}

\section{Partition functions of edge-coloring models}\label{sec:edge}
In this section we state and prove a generalization of Theorem \ref{thm:part edge}.
It is along the same lines as the generalization of Theorem \ref{thm:part spin} in the previous section.
The proof also goes along the same line, but as we will see below there are some details that are different.

\subsection{Partition functions for vertex-colored graphs}
Let $G=(V,E)$ be a graph. 
Suppose that we have $k$-color edge-coloring models $h^v$ for each $v\in V$. Let us write $\mathcal{H}=(h^v)_{v\in V}$.
Often the pair $(G,\mathcal H)$ is called a \emph{signature grid}, cf. \cite{CHL10,CLX11,CGW13}.
Then we can extend the definition of the partition function of an edge-coloring model as follows:
\begin{equation}
p(G)(\mathcal{H})=\sum_{\phi:E\to [k]}\prod_{v\in V} h^{v}(\phi(\delta(v))).
\end{equation}
We will refer to $p(G)(\mathcal{H})$ as the \emph{partition function} of $\mathcal{H}$. 
Clearly, if all $h^v$ are equal we obtain the ordinary partition function of an edge-coloring model.
It is also called the \emph{Holant problem} of the signature grid $(G,\mathcal{H})$ cf. \cite{CHL10,CLX11,CGW13}.
We have the following result, which implies Theorem \ref{thm:part edge}.
\begin{theorem}\label{eq:holant}
Let $\Delta,k\in \N$. Then there exists a deterministic algorithm, which, given a graph $G=(V,E)$ of maximum degree at most $\Delta$, $k$-color edge-coloring models $h^v$, $v\in V$, that satisfy $|h^v(\phi)-1|\leq 0.35/(\Delta+1)$ for all $\phi\in \N^k$ and $v\in V$, and $\eps>0$, computes a multiplicative $\eps$-approximation to $p(G)(\mathcal{H})$ in time $(|V|/\eps)^{O(1)}$.
\end{theorem}
\begin{rem}
The constant $0.35$ may be replaced by $0.47$ if $\Delta\geq 3$ and by $0.56$ if $\Delta$ is large enough; see \cite{R15}.
Moreover, for readers familiar with the orthogonal group invariance of these partition functions one can use Corollary 6b from \cite{R15} to find a larger family of edge-coloring models for which the partition function can be efficiently approximated.
\end{rem}
\begin{proof}
Let $J$ denote the constant ones function $J:\N^k\to \C$ (defined by $J(\phi)=1$ for all $\phi\in \N^k$). Let for $z\in \C$, $g^v(z):=J+z(h^v-J)$ and let $\mathcal{H}(z):=(g^v(z))_{v\in V}$.
Consider the following univariate polynomial:
\begin{equation}
q(G)(z):=k^{-|E|}p(G)(\mathcal{H}(z)).\label{eq:define pol}
\end{equation}
Observe that $q(G)(1)=k^{-|E|}p(G)(\mathcal H)$ and that $q(G)$ is a polynomial of degree at most $n:=|V|$.
So, just as in the previous section, the problem of approximating the partition function $p(G)(\mathcal H)$ is replaced by approximating an evaluation of a univariate polynomial.

By Corollary 6a from \cite{R15} (which is valid for multigraphs) there exists $\delta>0$ such that $q(G)(z)\neq 0 $ whenever $|z|\leq 1+\delta$.
We will show (in Theorem~\ref{thm:compute coef frag}) that for any $n$-vertex graph $G$ of maximum degree at most $\Delta$, we can compute the first $m$ inverse power sums of $q(G)$ in time $(n/ \varepsilon)^{O(1)}$, where $m = C\ln(n / \varepsilon)$ and $C = C(1, 1+ \delta)$ is the constant in Corollary~\ref{cor:approx}.
Noting that the degree of $q(G)$ is at most $n$, 
Corollary~\ref{cor:approx} implies we can compute a multiplicative $\varepsilon$-approximation to $q(G)(1)$ in time $(n/ \varepsilon)^{O(1)}$. 

Ideally we would like to do this using Theorem \ref{thm:compute coef} just as in the proof of Theorem \ref{thm:edge color spin}.
Since partition functions of edge-coloring models are multiplicative, the polynomial $q$ is also multiplicative and it has constant term equal to $1$.
So to be able to apply Theorem \ref{thm:compute coef} we need only check that the coefficients of $q$ can be expressed as linear combinations of (colored) induced graph counts.
This is in fact proved in \cite{R15} if all $h^v$ are equal, but there it is not clear whether the coefficients $\lambda_{H,i}$ in \eqref{eq:ind coef} can be computed efficiently.
So instead of directly applying Theorem \ref{thm:compute coef} we will have to do a little more work, which we postpone to the next section. 
\end{proof}

\subsection{Computing coefficients of $q(G)(z)$}\label{subsec:coef}
By definition,  
\begin{align}
q(G)(z) &= k^{-|E|}\sum_{\phi:E\to[k]}\prod_{v\in V}(J+ z(h^{v} -J))(\phi(\delta(v)))\nonumber
\\
&= k^{-|E|}
\sum_{i=0}^n z^i\bigg( \sum_{\substack{U\subseteq V\\|U|=i}}\sum_{\phi:E\to[k]}\prod_{u\in U}(h^{u}-J)(\phi(\delta(u)))\bigg)\nonumber
\\
 &=k^{-|E|}
\sum_{i=0}^n z^i\bigg( \sum_{\substack{U\subseteq V\\|U|=i}}k^{|E\setminus E(U)|}\sum_{\phi:E(U)\to[k]}\prod_{u\in U}(h^{u}-J)(\phi(\delta(u)))\bigg),
\label{eq:pf G(U)}
\end{align}  
where for $U\subseteq V$, $E(U)$ denotes the set of edges of $G$ that are incident with at least one vertex of $U$.

The second sum inside the brackets of the third line of \eqref{eq:pf G(U)} is almost the partition function of $(h^u-J)_{u\in U}$, except that the pair $(U,E(U))$ may not actually be a graph as some of the edges of $E(U)$ are not spanned by $U$. 
We will refer to such a graph-like structure as a \emph{fragment} and the edges in $E(U)$ that are `sticking out' (i.e. not spanned by $U$) as \emph{half edges}. 
Formally, a \emph{fragment}, is a pair $(H,\kappa)$, where $H$ is a vertex-colored graph and where $\kappa$ is a map $\kappa:V(H)\to \{0,1,\ldots,\Delta\}$, which records the number of half edges incident with each vertex. 
Suppose $G$ has $n$ vertices labeled $1,\ldots,n$. From now on we will consider the graph $G$ as a vertex-colored graph, where vertex $i$ gets color $i$ for $i=1,\ldots,n$.
For $U\subseteq V$ we let $G(U)$ be the fragment $(G[U],\kappa)$ where $\kappa(u)$ is equal to the number of edges that connect $u$ with $V\setminus U$.
Note that the graph $G$ itself can be thought of as a fragment by taking the map $\kappa:V(G)\to \{0,\ldots,\Delta\}$ to be $\kappa(v)=0$ for all $v\in V(G)$.

Clearly, for each $U$ of size $i$ the expression inside the second sum of the third line of \eqref{eq:pf G(U)} only depends on the isomorphism class of the fragment $G(U)$. (An \emph{isomorphism} from a fragment $(H,\kappa)$ to a fragment $(H',\kappa')$ is an isomorphism $\alpha$ of the underlying vertex-colored graphs and such that for each $u\in V(H)$, $\kappa(u)=\kappa'(\alpha(u))$.)
For a fragment $F=(H,\kappa)$ let $E(F)$ denote the set of edges of $F$ including half edges and let $V(F)$ denote the vertex set of the underlying graph $H$.
Assume that for each $i=1,2,\ldots$ we have a $k$-color edge coloring-model $h^i$ and let $\mathcal{H}=(h^i)_{i\geq 1}$ Then define, 
\begin{equation}
p(F)(\mathcal{H}):=\sum_{\phi:E(F)\to [k]}\prod_{v\in V(F)}h^{v}(\phi(\delta(v))).
\end{equation}
Here we implicitly identify the the color of a vertex of $H$ with the vertex itself.
Define for a fragment $F=(H,\kappa)$, $\ind^*(F,G)$ to be the number of sets $U$ of size $|V(F)|$ such that $G(U)$ is isomorphic to $F$.
Writing $\mathcal H-J=(h^i-J)_{i\geq 1}$, we can rewrite \eqref{eq:pf G(U)} as
\begin{align}
q(G)(z) &=k^{-|E|}\sum_{i=0}^nz^i\bigg(\sum_{\substack{F=(H,\kappa)\\ |V(H)|=i}} k^{|E|-|E(F)|}p(F)(\mathcal H-J)\ind^*(F,G)\bigg) \nonumber
\\
&=\sum_{i=0}^nz^i\bigg(\sum_{\substack{F=(H,\kappa)\\ |V(H)|=i}} k^{-|E(F)|}p(F)(\mathcal H-J)\ind^*(F,G)\bigg),\label{eq:fragment}
\end{align}
where the sum runs over fragments. This shows how to extend $q$ to vertex-colored graphs.
Let us denote the coefficient of $z^i$ in \eqref{eq:fragment} by $e_i$.
In \cite{R15} it is proved that in case all $h^i$ are equal, $\ind^*(F,G)$ can be expressed as a linear combination of the parameters $\ind(H,G)$ for certain graphs $H$. 
As mentioned above, the coefficients in this expression may not be easy to compute (at least we do not know how to do this).
So we will have to work with the parameters $\ind^*(F,\cdot)$ instead.
This is not a severe problem, since essentially if we replace $\ind$ in \eqref{eq:ind coef} by $\ind^*$, then Theorem \ref{thm:compute coef} remains valid.
Indeed, we have the following theorem.
\begin{theorem}\label{thm:compute coef frag}
Let $C>0$ and $\Delta\in \N$.
Then there is a deterministic $(n/\eps)^{O(1)}$-time algorithm, which, 
given any $n$-vertex graph $G$ of maximum degree at most $\Delta$ and any $\eps>0$, computes the inverse power sums $p_1,\ldots,p_m$ of $q(G)$ for $m=C\ln(n/\eps)$.
\end{theorem}

The proof of Theorem \ref{thm:compute coef frag} follows the same line as the proof of Theorem \ref{thm:compute coef}.
Essentially we need to replace graphs by fragments in the proof and check that everything remains valid.
For completeness we will give the proof.

We first need to note that for a fragment $F_1=(H_1,\kappa_1)$ the graph parameter $\ind^*(F_1,\cdot)$ can be extended to the collection of all fragments as follows: for a fragment $F_2=(H_2,\kappa_2)$ we let $\ind^*(F_1,F_2)$ denote the number of sets $S\subseteq V(H_2)$ such that $H_1$ is isomorphic to $H_2[S]$ as vertex-colored graphs and such that for each vertex $v$ of $H_1$ we have that the number of neighbours of $v$ in $V(H_2)\setminus S$ is equal to $\kappa_1(v)-\kappa_2(v)$.
Then for two fragments $F_1$ and $F_2$ we have
\begin{equation}
\ind^*(F_1,\cdot)\cdot\ind^*(F_2,\cdot)=\sum_{F} c^F_{F_1,F_2} \ind^*(F,\cdot),	\label{eq:prod ind fragment}
\end{equation}
where the sum runs over all fragments $F$ and where for a fragment $F$, $c^F_{F_1,F_2}$ denotes the number of pairs of subsets $(S,T)$ of $V(F)$ such that $S\cup T=V(F)$ and $F_1=F(S)$ and $F_2=F(T)$. (Here $F(S)$ is the fragment induced by $S$, i.e., if $F=(H,\kappa)$, then $F(S)=(H[S],\alpha)$ where for $s\in S$ we set $\alpha(s)=\deg_H(s)-\deg_{H[S]}(s) +\kappa(s)$.)
We call a fragment $F=(H,\kappa)$ \emph{connected} if the graph $H$ is connected.
We now adapt some of the statements and proofs of the results in Section \ref{sec:coef} to include fragments.

We start with some definitions.
By $\F$ we denote the collection of all fragments and by $\F_k$ for $k\in \N$ we denote the collection of fragments with at most $k$ vertices. 
For two fragments $F_1=(H_1, \kappa_1)$ and $F_2=(H_2, \kappa_2)$, 
$F_1 \cup F_2:= (H, \kappa)$, where $H = H_1 \cup H_2$ and 
$\kappa: V(H_1 \cup V(H_2) \to \{0,1, \ldots, \Delta\}$ is the 
map whose restriction to $V(H_1)$ is $\kappa_1$ and whose 
restriction to $V(H_2)$ is $\kappa_2$.  
An \emph{invariant of fragments} is a function $f:\F\to S$ for some set $S$ that takes the same value on isomorphic fragments.
Call an invariant of fragments $f$ \emph{multiplicative} if $f(\emptyset)=1$ and  $f(F_1\cup F_2)=f(F_1)f(F_2)$ for all fragments $F_1,F_2$ .
The maximum degree of a fragment $F=(H,\kappa)$ is equal to the maximum of $\deg(v)+\kappa(v)$ over $v\in V(G)$.

\begin{lemma}\label{lem:connected count frag}
Let $F=(H,\kappa)$ be a connected fragment on $k$ vertices and let $\Delta\in \N$. Then there is an $O(k^2n^{2}\Delta^{2(k-1)})$-time algorithm, which, given any $n$-vertex fragment $\hat F$ with maximum degree at most $\Delta$, computes the number $\ind^*(F,\hat F)$.
\end{lemma}
Note that Lemma~\ref{lem:connected count frag} enables us to test for isomorphism of fragments between bounded degree fragments when $|V(F)| = |V(\hat{F})|$.
\begin{proof}
This follows immediately from the proof of Lemma \ref{lem:connected count}.
We apply the proof of Lemma \ref{lem:connected count} to the underlying graphs and then remove any potential embedding that either violates the vertex coloring constraints or the constraints that $\kappa$ imposes.
\end{proof}

We call an invariant of fragment $f:\F\to \C$ \emph{additive} if for each $F_1,F_2\in \F$ we have  $f(F_1\cup F_2)=f(F_1)+f(F_2)$. 
The following variation of a lemma due to Csikv\'ari and Frenkel \cite{CF12} has exactly the same proof as Lemma \ref{lem:additivity}; one just needs to replace graph by fragment everywhere in the proof.
\begin{lemma}\label{lem:additivity frag}
Let $f:\F\to \C$ be an invariant of fragments given by $f(\cdot):=\sum_{F\in \F}a_F\ind^*(F,\cdot)$. 
Then $f$ is additive if and only if $a_F=0$ for all fragments $F$ that are disconnected. 
\end{lemma}

We now sketch the proof of Theorem \ref{thm:compute coef frag}.
\subsubsection{Proof of Theorem \ref{thm:compute coef frag}}
Let $\zeta_1,\ldots,\zeta_d\in \C$ be the roots of the polynomial $q(G)$ and recall that for $\ell\in \N$, $p_\ell$ is the $\ell$th inverse power sum of the $\zeta_i$. Here $d$ denotes the degree of $q(G)=\sum_{i=0}^d e_iz^i$, which is at most $n$.
By \eqref{eq:fragment}, for $i\geq 1$, the $e_i$ can be expressed as linear combinations of induced fragments counts of fragments with at most $\ell$ vertices. Since $e_1=-p_1$, this implies that the same holds for $p_1$.
By induction, \eqref{eq:prod ind fragment} and \eqref{eq:newton} (using that $e_0=1$) we have that for each $\ell$
\begin{equation}\label{eq:power to ind frag}
p_\ell=\sum_{F\in \F_{\ell}}a_{F,\ell}\ind^*(F,G),
\end{equation}
for certain, yet unknown, coefficients $a_{F,\ell}$.

Since $q$ is multiplicative, the inverse power sums are additive. Thus Lemma \ref{lem:additivity frag} implies that $a_{F,\ell}=0$ if $F$ is not connected. 
Denote by $\mathcal{C'}_{\ell}(G)$ the set of connected fragments $F$ of order at most $\ell$ such that $\ind^*(F,G)\neq 0$.
This way we can rewrite \eqref{eq:power to ind frag} as follows:
\begin{equation}\label{eq:power support frag}
p_\ell=\sum_{F\in \mathcal{C'}_\ell(G)}a_{F,\ell}\ind^*(F,G).
\end{equation}
The next lemma says that we can compute the coefficients $a_{F,\ell}$ efficiently for $\ell=1,\ldots,m$, where $m=C\ln(n/ \eps)$.
\begin{lemma}\label{lemma:compute frag}
There is an $O(n/\eps)^{O(1)}$-time algorithm, which given an $n$-vertex graph $G$ and $\varepsilon > 0$, computes and lists the coefficients $a_{F,\ell}$ in \eqref{eq:power support frag} for all $F\in \mathcal{C'}_{\ell}(G)$ and all $\ell=1,\ldots,m = C \ln(n / \varepsilon)$.
\end{lemma}
\begin{proof}
Using the algorithm of Lemma \ref{lem:enumerate}, we first compute the sets $\mathcal{T}_{\ell}$ consisting of all subsets $S$ of $V(G)$ such that $|S|\leq \ell$ and $G[S]$ is connected, for $\ell=1\ldots, m$.
This takes time bounded by $(n/\eps)^{O(1)}$.
The collection $\mathcal{C'}_{\ell}(G)$ can be obtained from $\mathcal{T}_\ell$ by looking for each $S\in \mathcal{T}_\ell$ and each $v\in S$ how many neighbours $v$ has in $G\setminus S$.
So the total time to compute and list the $\mathcal{C'}_{\ell}(G)$ is bounded by $(n/\eps)^{O(1)}$.

To prove the lemma, let us fix $\ell\leq m$ and show how to compute the coefficients $a_{F,\ell}$, assuming that  we have already computed and listed the coefficients $a_{F,\ell'}$ for all $\ell'<\ell$.
Let us fix $F\in \mathcal{C'}_{\ell}(G)$.
By the identities (\ref{eq:newton}), it suffices to compute the coefficient of $\ind^*(F,\cdot)$ in $e_{i}p_{\ell-i}$ for $i=1,\ldots,\ell$ (where we set $p_0=1)$.
By \eqref{eq:fragment}, \eqref{eq:prod ind fragment} and \eqref{eq:power to ind frag} we know that the coefficient of $\ind^*(F,\cdot)$ in $e_{i}p_{\ell-i}$ is given by
\begin{align*}\label{eq:coef ep frag}
&\sum_{\substack{F_ 1\in \F_i\\|V(F_1)|=i}}\sum_{F_2\in \F_{\ell-i}} c^F_{F_1,F_2}a_{F_2,(\ell-i)}\cdot \frac{p(F_1)(\mathcal{H}-J)}{k^{|E(F_1)|}}
\\
=&\sum_{\substack{S,T\subseteq V(F)\\ S \cup T = V(F)\\|S|=i, |T|\leq 
\ell-i}} a_{F(T),(\ell-i)}\cdot \frac{p(F(S))(\mathcal{H}-J)}{k^{|E(F(S))|}}.
\end{align*}
For each such pair $(S,T)$, we need  to compute $p(F(S))(J-\mathcal{H})k^{-|E(F(S))|}$ and look up $ a_{F(T), (\ell-i)}$.
We can compute $p(F(S))(J-\mathcal{H})k^{-|E(F(S))|}$ in time bounded by $O(k^{\Delta \ell})=(n/\eps)^{O(1)}$.

Looking up $a_{F(T),(\ell-i)}$ in the given list requires us to test isomorphism of $F(T)$ with each fragment in $\mathcal{C'}_{\ell-i}(G)$ (noting that $a_{F(T),(\ell-i)}=0$ if $F(T) \not \in \mathcal{C'}_{(\ell-i)}(G)$ by Lemma~ \ref{lem:additivity frag}). Using Lemma~\ref{lem:connected count frag} to test for isomorphism, this takes time at most 
\[O(|\mathcal{C'}_{(\ell-i)}(G)|(\ell-i)^2\Delta^{2(\ell-i-1)})=O(n/\eps)^{O(1)}.
\]
Here we use Lemma~\ref{lem:graph count} to bound $|\mathcal{C'}_{(\ell-i)}(G)|\leq |\mathcal{T}_{(\ell-i)}(G)|$.
Together, all this implies that the coefficient of $\ind^*(F,\cdot)$ in $p_{\ell-i}e_{i}$ can be computed in time bounded by $(n/\eps)^{O(1)}$, and so the coefficient $a_{F,\ell}$ can be computed in time $(n/\eps)^{O(1)}$.
Thus all coefficients $a_{F,\ell}$ for $F\in \mathcal{C'}_{\ell}(G)$ can be computed and listed in time bounded by  $|\mathcal{C'}_{\ell}(G)|(n/\eps)^{O(1)}=(n/\eps)^{O(1)}$.
This can be done for each $\ell = 1, \ldots, m$ in time $(n/\eps)^{O(1)}$.
\end{proof}
To finish the proof of the theorem, we compute $p_\ell$ for each $\ell=1,\ldots, m$ by adding all the numbers $a_{F,\ell}\ind^*(F,G)$ over all $F\in \mathcal{C'}_{\ell}(G)$. This can be done in time 
\[
O(m|\mathcal{C'}_{m}(G)|n^2\Delta^{2(m-1)})=(n/\eps)^{O(1)},
\] 
where we have used that computing $\ind^*(F,G)$ with $F\in \mathcal{C'}_{\ell}(G)$ takes time bounded by $O(n^2\Delta^{2(m-1)})$ by Lemma~\ref{lem:connected count frag}.
This finishes the proof.

\section{Concluding remarks and open questions}\label{sec:conclude}
In this paper we have presented a direct connection between the absence of complex roots for a large class of graph polynomials (BIGCPs) and the existence of (deterministic) algorithms to efficiently approximate evaluations of these polynomials.
We have illustrated its use by giving deterministic polynomial-time approximation algorithms for evaluations of the Tutte polynomial, the independence polynomial and graph polynomials obtained from spin and edge-coloring models at complex numbers on bounded degree graphs. 

As is noted in the introduction Theorem~\ref{thm:ind general} does not allow us to efficiently approximate the independence polynomial at $\lambda$ for $\lambda^*\leq \lambda <\lambda_c$, while this can be done with the correlation decay approach cf. Weitz \cite{W6}.
However, confirming a conjecture of Sokal \cite{S1}, Peters and the second author \cite{PR17} proved the following:
\begin{theorem}\label{ques:sokal}
Let $\eps>0$ and $\Delta\in \N$. Then there exists $\delta>0$ such that for any graph $G=(V,E)$ of maximum degree at most $\Delta$, and $\lambda_v\in \C$ for $v\in V$ that satisfy
\begin{equation}
|\Im (\lambda_v)|\leq \delta\text{, and }0\leq \Re(\lambda)\leq (1-\eps) \frac{(\Delta-1)^{\Delta-1}}{(\Delta-2)^\Delta}
\end{equation}
for all $v\in V$, we have $Z(G)((\lambda_v)_{v\in V})\neq0$.
\end{theorem}
Now combining this result with the approach in Section \ref{subsec:clawfree}, it follows that with the methods in this paper we can efficiently approximate the independence polynomial at $\lambda$ for $\lambda <\lambda_c$, thereby giving a different proof of Weitz's result.

Let us restate another conjecture of Sokal \cite[Conjecture 21]{J06}, which, if true, would by the methods of the present paper imply that we have an efficient algorithm for approximately counting the number of $(\Delta+1)$-colorings in any graph of maximum degree at most $\Delta$.
\begin{ques}
Let $\Delta\in \N$. Is it true that $Z_T(G)(-1,q)\neq0$ for every $q\in \C$ satisfying $\Re(q)>\Delta$ and every graph $G$ of maximum degree at most $\Delta$?
\end{ques}

This connection between absence of complex roots and efficient approximation algorithms naturally leads to the question of how hard it is to approximate evaluations of these graph polynomials close to (complex) roots. 
In light of this we remark that some progress on this question has been made.
As mentioned in the introduction, there exists a sequence of trees $T_n$ of maximum degree at most $\Delta$ and $\lambda_n<-\lambda^*$ with $\lambda_n\to -\lambda^*$ such that $Z(T_n,\lambda_n)=0$. This was utilized by Galanis, Goldberg and \v{S}tefankovi\v{c}, to show that it is NP hard to approximate $Z_G(\lambda)$ when $\lambda<-\lambda^*(\Delta)$.

Another question that arises naturally is the following. Barvinok \cite{B14a,B16} found quasi-polynomial-time approximation algorithms for computing the permanent of certain matrices, based on absence of zeros. 
Our method for computing inverse power sums of BIGCPs on bounded degree graphs presented in Section \ref{sec:coef} does not seem to apply to permanents of general matrices. 
It would be very interesting to find a more general method that also applies to permanents.

Our algorithmic results in Section~\ref{sec:coef} can be interpreted as giving a fixed parameter tractability result for determining $\ind(H,G)$ for certain graphs $H$. 
If $G$ has bounded degree, the algorithm runs in time $2^{O(|V(H)|}|V(G)|^{O(1)}$.
However, the algorithm only works for graphs $H$ for which $\ind(H,\cdot)$ are coefficients of a multiplicative graph polynomial. Very recently, we were able to extend the algorithm to all graphs $H$; see \cite{PaR17}.
A natural question is whether our approach can be extended to other classes of graphs such as planar graphs for example. More concretely, let us state the following question.
\begin{ques}
Is there an algorithm, which given a planar (or more generally, bounded genus) graph $G$ and $k\in \N$, returns the number of independent sets of size $k$ of $G$ in time bounded by $2^{O(k)}|V(G)|^{O(1)}$? 
\end{ques}

\section*{Acknowledgements}
We thank Alexander Barvinok for stimulating discussions, useful remarks and for sharing the results in \cite{B16PC} with us.
We thank Andreas Galanis for informing us about \cite{Sr15}. 
We are grateful to Pinyan Lu for some useful remarks on an earlier version of this paper.

We moreover thank the anonymous referees for helpful comments and suggestions, improving the presentation of the paper.

\end{document}